\documentclass[12pt,a4paper]{article}
\usepackage{amsmath,amssymb,amsbsy,amscd,amsthm}

\newtheorem{theorem}{Theorem}[section]
\newtheorem{proposition}[theorem]{Proposition}
\newtheorem{lemma}[theorem]{Lemma}
\newtheorem{corollary}[theorem]{Corollary}
\newtheorem{problem}[theorem]{Problem}

\newcommand{\nd }{\noindent}
\newcommand{\T }{{\rm T}}
\newcommand{\Aut }{\mathop{\rm Aut}\nolimits}
\newcommand{\lcm}{{\rm lcm}}

\newcommand{\mdeg }{\mathop{\rm mdeg}\nolimits}
\newcommand{\w}{{\bf w}}
\newcommand{\degw}{\deg_{\w}}
\newcommand{\mdegw}{\mdeg_{\w}}
\newcommand{\vv}{{\bf v}}
\newcommand{\degv}{\deg_{\vv}}
\newcommand{\mdegv}{\mdeg_{\vv}}
\newcommand{\id	}{{\rm id}}
\newcommand{\zs}{\{ 0\} }
\newcommand{\sm}{\setminus}
\newcommand{\N}{{\bf N}}
\newcommand{\Z}{{\bf Z}}
\newcommand{\Q}{{\bf Q}}
\newcommand{\Zn}{{\N _{0}}}

\newcommand{\x}{{\bf x}}
\newcommand{\kx}{k[{\bf x}]}
\newcommand{\kxr}{k({\bf x})}

\begin{document}

\title{On the Kara\'s type theorems for the multidegrees 
of polynomial automorphisms}

\author{Shigeru Kuroda}

\date{}


\footnotetext{Partly supported by the Grant-in-Aid for 
Young Scientists (B) 24740022, 
Japan Society for the Promotion of Science. }

\maketitle

\begin{abstract}
To solve Nagata's conjecture, 
Shestakov-Umirbaev 
constructed a theory 
for deciding wildness 
of polynomial automorphisms in three variables. 
Recently, 
Kara\'s and others study multidegrees of 
polynomial automorphisms 
as an application of this theory. 
They give various necessary conditions 
for triples of positive integers to be multidegrees 
of tame automorphisms in three variables. 
In this paper, 
we prove a strong theorem unifying these results 
using the generalized Shestakov-Umirbaev theory. 
\end{abstract}

\section{Introduction}

Let $k$ be a field of characteristic zero, 
and $\kx =k[x_1,\ldots ,x_n]$ 
the polynomial ring in $n$ variables over $k$. 
In this paper, 
we discuss the group 
$\Aut _k\kx $ of automorphisms of the 
$k$-algebra $\kx $. 
To denote an element of $\Aut _k\kx $, 
we often use the notation $F=(f_1,\ldots ,f_n)$, 
where each $f_i$ represents the image of $x_i$ 
under the automorphism. 
We say that $F$ is {\it elementary} if there exist 
$1\leq l\leq n$, 
$a\in k^{\times }$ and 
$p\in k[\x \sm \{ x_l\} ]$ 
such that $f_l=ax_l+p$ and $f_i=x_i$ 
for each $i\neq l$, 
where we regard $\x $ 
as the set $\{ x_1,\ldots ,x_n\} $ of variables. 
The subgroup $\T_n(k)$ of $\Aut _k\kx $ 
generated by all the elementary automorphisms 
is called the {\it tame subgroup}. 
We say that $F$ 
is {\it tame} 
if $F$ belongs to $\T_n(k)$, 
and {\it wild} otherwise.

The following problem is one of the fundamental problems 
in Polynomial Ring Theory.

\begin{problem}\label{prob:TGP}
Does it hold that $\Aut _k\kx =\T _n(k)$? 
\end{problem}

Since $k$ is a field, 
every automorphism of $\kx $ is elementary if $n=1$. 
Due to Jung~\cite{Jung}, 
$\Aut _k\kx =\T _2(k)$ holds for $n=2$. 
In the case of $n=3$, 
Nagata~\cite{Nagata} conjectured that 
$\Aut _k\kx \neq \T_3(k)$, 
and constructed a candidate of wild automorphism. 
Recently, 
Shestakov-Umirbaev~\cite{SU'}, \cite{SU} solved 
the conjecture in the affirmative. 
They not only proved that Nagata's automorphism is wild, 
but also gave a criterion for deciding wildness 
of elements of $\Aut _k\kx $ when $n=3$. 
Problem~\ref{prob:TGP} remains open for $n\geq 4$.

Now, 
for each $F\in \Aut _k\kx $, 
we define the {\it multidegree} of $F$ by 
$$
\mdeg F=(\deg f_1,\ldots ,\deg f_n). 
$$
Here, 
$\deg f$ denotes the total degree 
of $f$ for each $f\in \kx $. 
We set 
$$
\mdeg S=\{ \mdeg F\mid F\in S\} 
$$
for each subset $S$ of $\Aut _k\kx $.

The following proposition is due to 
Kara\'s~\cite[Proposition 2.2]{345}.

\begin{proposition}[Kara\'s]\label{prop:Karas345}
Let $d_1,\ldots ,d_n\in \N $ be such that 
$d_1\leq \cdots \leq d_n$. 
If there exists $2\leq i\leq n$ 
such that $d_i$ belongs to 
$$
\Zn d_1+\cdots +\Zn d_{i-1}
:=\{ u_1d_1+\cdots +u_{i-1}d_{i-1}
\mid u_1,\ldots ,u_{i-1}\in \Zn \} , 
$$
then $(d_1,\ldots ,d_n)$ belongs to 
$\mdeg \T _n(k)$. 
\end{proposition}

Here, 
$\N $ and $\Zn $ denote the sets of 
positive integers and 
nonnegative integers, 
respectively. 
In the same paper, 
Kara\'s proved that 
$(3,4,5)$, $(3,5,7)$, $(4,5,7)$ or $(4,5,11)$ 
do not belong to 
$\mdeg \T _3(k)$ using 
the theory of Shestakov-Umirbaev~\cite{SU'}, \cite{SU}. 
This means  that any automorphism with multidegree 
$(3,4,5)$, $(3,5,7)$, $(4,5,7)$ or $(4,5,11)$ is wild.

Following this paper, 
Kara\'s and others recently study detailed conditions 
for elements of $\N ^3$ to belong to $\mdeg \T _3(k)$. 
They derive from the Shestakov-Umirbaev theory 
various conditions under which 
$(d_1,d_2,d_3)\in \N ^3$ belongs to $\mdeg \T_3(k)$ 
if and only if $d_2$ is divisible by $d_1$, 
or $d_3$ belongs to $\Zn d_1+\Zn d_2$. 
Because of Proposition~\ref{prop:Karas345}, 
this is the same as giving conditions under which 
$(d_1,d_2,d_3)\in \N ^3$ does not belong to $\mdeg \T_3(k)$ 
if $d_2$ is not divisible by $d_1$, 
and $d_3$ does not belong to $\Zn d_1+\Zn d_2$. 
By a {\it Kara\'s type theorem}, 
we mean a theorem which describes such a condition. 
In the following diagram, 
the oval represents the set of 
elements of $\N ^3$ satisfying one of such conditions.

{\unitlength=5mm 
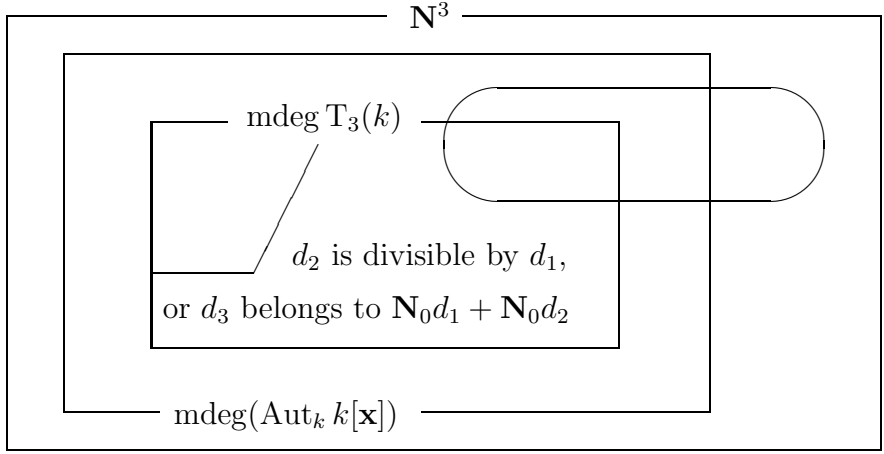
\begin{picture}(20,14)

\put(2.7,7){

\put(0,-0.3){

\put(2.3,-3)
{\line(1,0){12.3}\line(0,1){6}}
\put(2.3,3){\line(0,-1){0.5}\line(1,0){2}}
\put(2.3,-3){\line(0,1){5.5}}
\put(9.4,3){\line(1,0){5.2}}

\put(6.7,2.4){\line(-1,-2){1.7}}
\put(5,-1){\line(-1,0){2.7}}

\put(4.8,2.9){$\mdeg \T_3(k)$}
\put(6,-0.8){$d_2$ is divisible by $d_1$, }
\put(2.6,-2.2){or $d_3$ belongs to $\Zn d_1+\Zn d_2$}

}

\put(15,2.1){\oval(10,3)}

\put(0,-5){\line(1,0){2.5}}
\put(9.4,-5){\line(1,0){7.6}}
\put(17,4.5){\line(0,-1){9.5}}
\put(0,4.5){\line(0,-1){9.5}\line(1,0){17}}
\put(2.9,-5.3){$\mdeg (\Aut _k\kx )$}

\put(-1.5,-6){\line(1,0){23}}
\put(-1.5,-6){\line(0,1){11.5}}
\put(21.5,-6){\line(0,1){11.5}}
\put(-1.5,5.5){\line(1,0){10}}
\put(10.5,5.5){\line(1,0){11}}

\put(9.1,5.2){$\N ^3$}

}

\end{picture}}

Since many partial results were published separately, 
we consider it necessary 
to prove a strong theorem unifying them. 
The purpose of this paper 
is to interpret the generalized 
Shestakov-Umirbaev theory~\cite{SUineq}, \cite{tame3}
in terms of multidegrees. 
This naturally leads to a very general version 
of Kara\'s type theorem (Theorem~\ref{thm:main}) 
which is also valid for the case of ``weighted" degree. 
In the case of total degree, 
our result implies the following theorem.

For $d_1,d_2,d_3\in \N $, 
consider the conditions (a), (b) and (c) as follows: 

\medskip 

\nd (a) One of the following holds: 

\smallskip 

\nd (a1) $3d_2\neq 2d_3$ and $sd_1\neq 2d_3$ 
for any odd number $s\geq 3$. 

\smallskip 

\nd (a2) $d_1+d_2\leq d_3+2$.

\medskip 

\nd (b) One of the following holds: 

\smallskip 

\nd (b1) $\gcd (d_1,d_2,d_3)=\gcd (d_1,d_2)\leq 3$. 

\smallskip 

\nd (b2) $d_1+d_2+d_3\leq \lcm (d_1,d_2)+2$. 

\medskip 

\nd (c) $d_2$ is not divisible by $d_1$, 
and $d_3$ does not belong to $\Zn d_1+\Zn d_2$. 

\medskip

Then, 
we have the following theorem.

\begin{theorem}\label{thm:total degree}
Let $d_1\leq d_2\leq d_3$ be positive integers 
satisfying {\rm (a)}, {\rm (b)} and {\rm (c)}. 
Then, 
$(d_1,d_2,d_3)$ does not belong to $\mdeg \T_3(k)$. 
\end{theorem}

We note that (a) holds if and only if 
(a1) or one of the following holds: 

\medskip 

\nd (a2$'$) $3d_2=2d_3$ and $2d_1\leq d_2+4$. 

\smallskip 

\nd (a2$''$) 
$sd_1=2d_3$ and $2d_2\leq (s-2)d_1+4$ 
for some odd number $s\geq 3$. 

\medskip 

Actually, 
(a2) is equivalent to $2d_1\leq d_2+4$ 
if $3d_2=2d_3$, 
and to $2d_2\leq (s-2)d_1+4$ 
if $sd_1=2d_3$ for some odd number $s\geq 3$. 
The condition (b1) is equivalent to 
$\gcd (d_1,d_2)\leq 3$ and $\gcd (d_1,d_2)\mid d_3$, 
which is fulfilled 
whenever $\gcd (d_1,d_2)=1$. 
By (c), 
Theorem~\ref{thm:total degree} 
is considered as a Kara\'s type theorem. 
As will be discussed in Section~\ref{sect:appl}, 
this theorem implies all the 
Kara\'s type theorems which the author knows, 
except for Kara\'s~\cite[Theorem 1]{456}. 
This result of Kara\'s can be derived 
from Theorem~\ref{thm:main} 
by assuming the lower bound for the degrees of 
``Poisson brackets" given in \cite[Theorem 14]{456}.

It is easy to see that, 
when $n\geq 3$, 
there exists $(d_1,\ldots ,d_n)\in \mdeg \T _n(k)$ 
with $d_1\leq \cdots \leq d_n$ for which 
$d_i$ does not belong to $\Zn d_1+\cdots +\Zn d_{i-1}$ 
for any $2\leq i\leq n$. 
For instance, 
take any prime numbers $p_1<\cdots <p_{n-1}$, 
and define 
$d_i=p_i^ip_{i+1}\cdots p_{n-1}$ for each $i\neq n$, 
and $d_n=(d_{n-1}-1)d_{n-2}+1$. 
Then, 
we have $d_1\leq \cdots \leq d_n$, 
and $d_i$ does not belong to $\Zn d_1+\cdots +\Zn d_{i-1}$ 
for any $2\leq i\leq n$, 
since $p_{i-1}\nmid d_i$. 
Moreover, 
$(d_1,\ldots ,d_n)$ 
is the multidegree of 
$F\in \T _n(k)$ defined by 
$f_i=x_i+x_n^{d_i}$ for each $i\neq n$, 
and $f_n=x_n+f_{n-2}^{d_{n-1}}-f_{n-1}^{d_{n-2}}$. 
Thus, 
it is also an interesting problem 
to find a good {\it sufficient} condition for elements of 
$\N ^n$ to belong to $\mdeg \T _n(k)$. 
We mention that 
the situation is more complicated 
in the case of weighted degree. 
In this case, 
a statement similar to Proposition~\ref{prop:Karas345} 
does not hold in general. 
It seems difficult to expect a simple criterion 
as in the case of total degree.

The structure of this paper is as follows. 
In Section~\ref{sect:main}, 
we state the main theorem 
after recalling basic definitions. 
Sections~\ref{sect:SUtheory}, \ref{sect:w} and 
\ref{sect:pf main} are devoted to the proof of 
the main theorem. 
In Section~\ref{sect:SUtheory}, 
we prove a key theorem using 
the generalized Shestakov-Umirbaev theory. 
In Section~\ref{sect:w}, 
we generalize Sun-Chen~\cite[Lemma 3.1]{SunChen} 
to the case of weighted degree. 
The proof of the main theorem is completed in 
Section~\ref{sect:pf main} using the results in 
Sections~\ref{sect:SUtheory} and \ref{sect:w}. 
In Section~\ref{sect:appl}, 
we prove a variety of Kara\'s type theorems 
as applications of our theorems.

\section{Main results}\label{sect:main}
\setcounter{equation}{0}

First, 
we recall basics of 
the weighted degrees 
of polynomial automorphisms and differential forms. 
Let $\Gamma $ be a 
{\it totally ordered additive group}, i.e., 
an additive group 
equipped with a total ordering such that 
$\alpha \leq \beta $ implies 
$\alpha +\gamma \leq \beta +\gamma $ 
for each $\alpha ,\beta ,\gamma \in \Gamma $. 
We denote by $\Gamma _+$ 
the set of positive elements of $\Gamma $. 
Let $\w =(w_1,\ldots ,w_n)$ 
be an $n$-tuple of elements of $\Gamma _+$. 
We define the $\w $-{\it weighted $\Gamma $-grading} 
$$\kx =\bigoplus _{\gamma \in \Gamma }\kx _{\gamma }$$ 
by setting $\kx _{\gamma }$ 
to be the $k$-vector subspace of $\kx $ 
generated by $x_1^{a_1}\cdots x_n^{a_n}$ 
for $a_1,\ldots ,a_n\in \Zn $ 
with $\sum _{i=1}^na_iw_i=\gamma $ 
for each $\gamma \in \Gamma $. 
Write $f\in \kx $ as 
$f=\sum _{\gamma \in \Gamma }f_{\gamma }$, 
where $f_{\gamma }\in \kx _{\gamma }$ 
for each $\gamma \in \Gamma $. 
Then, 
we define the $\w $-{\it degree} of $f$ by 
$$
\degw f=\max \{ \gamma \in \Gamma \mid f_{\gamma }\neq 0\} . 
$$
Here, 
the maximum of the empty set is defined to be $-\infty $. 
We define $f^{\w }=f_{\delta }$ if $\delta :=\degw f>-\infty $, 
and $f^{\w }=0$ otherwise. 
Then, 
we have 
\begin{equation}\label{eq:deg prop}
\degw fg=\degw f+\degw g\quad
\text{and}\quad (fg)^{\w }=f^{\w }g^{\w }
\end{equation}
for each $f,g\in \kx $. 
If $\Gamma =\Z $ and $\w =(1,\ldots ,1)$, 
then the $\w $-degree of $f$ 
is the same as the total degree of $f$.

We remark that, 
if $\degw f<w_i$ for $f\in \kx $ and $1\leq i\leq n$, 
then $f$ belongs to $k[\x \sm \{ x_i\} ]$. 
As a consequence of this fact, 
we obtain the following lemma.

\begin{lemma}\label{prop:deg F}
Let $\w \in (\Gamma _+)^n$ and $F\in \Aut _k\kx $ 
be such that 
$$
w_1\leq \cdots \leq w_n\quad 
\text{and}\quad 
\degw f_1\leq \cdots \leq \degw f_n. 
$$ 
Then, 
we have $\degw f_i\geq w_i$ for $i=1,\ldots ,n$. 
\end{lemma}
\begin{proof}
Suppose to the contrary that $\degw f_i<w_i$ for some $1\leq i\leq n$. 
Then, we have $\degw f_j<w_l$ 
for each $1\leq j\leq i$ and $i\leq l\leq n$. 
This implies that $f_1,\ldots ,f_i$ 
belong to $k[x_1,\ldots ,x_{i-1}]$. 
Hence, 
$f_1,\ldots ,f_i$ 
are algebraically dependent over $k$, 
a contradiction. 
\end{proof}

For each $F\in \Aut _k\kx $, 
we define the $\w $-{\it degree} and $\w $-{\it multidegree} 
by 
$$
\degw F=\sum _{i=1}^n\degw f_i\quad\text{and}
\quad 
\mdegw F=(\degw f_1,\ldots ,\degw f_n), 
$$
respectively. 
For each subset $S$ of $\Aut _k\kx $, 
we define 
$$
\mdegw S=\{ \mdegw F\mid F\in S\} . 
$$
Set 
$|\w |=w_1+\cdots +w_n$. 
Then, 
we know by Lemma~\ref{prop:deg F} 
that $\degw F\geq |\w |$, 
and $\degw F=|\w |$ if and only if 
$\mdegw F=(w_{\sigma (1)},\ldots,w_{\sigma (n)})$ 
for some $\sigma \in \mathfrak{S}_n$. 
Here, 
$\mathfrak{S}_n$ denotes the symmetric group 
of $\{ 1,\ldots ,n\} $.  
Thus, 
we may assume that $\degw F>|\w |$ 
in the study of $\mdegw F$. 
We mention that $F$ is tame if $\degw F=|\w |$ 
for some $\w \in (\Gamma _+)^n$ 
by \cite[Lemma 6.1 (i)]{tame3}.

The notion of the $\w $-degrees of differential forms 
is important 
in the study of polynomial automorphisms.
Let $\Omega _{\kx /k}$ 
be the module of differentials 
of $\kx $ over $k$, 
and $\bigwedge ^r\Omega _{\kx /k}$ 
the $r$-th exterior power of the 
$\kx $-module $\Omega _{\kx /k}$ 
for $r\in \N $. 
Then, 
each $\omega \in \bigwedge ^r\Omega _{\kx /k}$ 
is uniquely expressed as 
$$
\omega =\sum _{1\leq i_1<\cdots <i_r\leq n}
f_{i_1,\ldots ,i_r}dx_{i_1}\wedge \cdots \wedge dx_{i_r},
$$
where $f_{i_1,\ldots ,i_r}\in \kx $ for each $i_1,\ldots ,i_r$. 
Here, 
$df$ denotes the differential of $f$ for each $f\in \kx $. 
We define the $\w $-{\it degree} of $\omega $ by 
\begin{equation*}
\degw \omega =\max \{ \degw 
f_{i_1,\ldots ,i_r}x_{i_1}\cdots x_{i_r}\mid 
1\leq i_1<\cdots <i_r\leq n\} . 
\end{equation*}
When $\Gamma =\Z $ and $\w =(1,\ldots ,1)$, 
we denote $\degw \omega $ simply by $\deg \omega $. 
By definition, 
we have 
$$
\degw df\wedge dg=\max 
\left\{ 
\degw \left( 
\frac{\partial f}{\partial x_i}
\frac{\partial g}{\partial x_j}
-\frac{\partial f}{\partial x_j}
\frac{\partial g}{\partial x_i}
\right) x_ix_j
\Bigm|
1\leq i<j\leq n
\right\} 
$$
for $f,g\in \kx $. 
Hence, 
$\deg df\wedge dg$ is the same as 
the degree of the 
{\it Poisson bracket} $[f,g]$ 
defined by Shestakov-Umirbaev~\cite{SU'}.

For each $F\in \Aut _k\kx $, 
we have 
$df_1\wedge \cdots \wedge df_n=cdx_1\wedge \cdots \wedge dx_n$, 
where $c\in k^{\times }$ is the Jacobian of $F$. 
Hence, we get $\degw df_1\wedge \cdots \wedge df_n=|\w |$. 
Recall that 
$h_1,\ldots ,h_r$ 
are algebraically independent over $k$ 
if and only if $dh_1\wedge \cdots \wedge dh_r$ 
is nonzero 
for each $h_1,\ldots ,h_r\in \kx $ with $r\in \N $ 
(cf.~\cite[Section 26]{Matsumura}). 
Thus, 
$f_1^{\w },\ldots ,f_n^{\w }$ 
are algebraically independent over $k$ 
if and only if $df_1^{\w }\wedge \cdots \wedge df_n^{\w }$ is nonzero, 
and hence if and only if 
\begin{equation}\label{eq:minimal autom}
|\w |=\degw df_1\wedge \cdots \wedge df_n=\degw f_1+\cdots +\degw f_n
=\degw F. 
\end{equation}

Now, assume that $n=3$. 
We consider when 
$(d_1,d_2,d_3)\in (\Gamma _+)^3$ 
does not belong to $\mdegw \T_3(k)$. 
Note that, 
if $d_1$ and $d_2$ 
are linearly dependent over $\Z $, 
then there exists unique $(u_1,u_2)\in \N ^2$ 
such that $\gcd (u_1,u_2)=1$ and $u_2d_1=u_1d_2$. 
Then, 
we can find $d\in \Gamma _+$ such that $d_i=u_id$ for $i=1,2$. 
In this notation, 
we define $\gcd (d_1,d_2)=d$ and $\lcm (d_1,d_2)=u_1u_2d$. 
By definition, 
these elements of $\Gamma _+$ 
are the same as 
the greatest common divisor 
and least common multiple if $\Gamma =\Z $. 
For $d,e\in \Gamma _+$, 
we define $\Delta _{\w }(d,e)$ 
to be the minimum among $\degw df_1\wedge df_2$ 
for $F\in \T_3(k)$ 
such that $\degw f_1=d$ and $\degw f_2=e$. 
Here, 
the minimum of the empty set is defined to be $\infty $. 
If $\Gamma =\Z $ and $\w =(1,1,1)$, 
then we have $\Delta _{\w }(d,e)<\infty $ 
for any $d,e\in \N $, 
since $(d,e,e)$ belongs to $\mdeg \T_3(k)$ 
by Proposition~\ref{prop:Karas345}. 
For $N_1,\ldots ,N_r\subset \Z $ 
and $e_1,\ldots ,e_r\in \Gamma $ 
with $r\in \N $, 
we define 
$$
N_1e_1+\cdots +N_re_r
:=\left\{ \sum _{i=1}^ru_ie_i\Bigm| 
u_i\in N_i\text{ \rm for }i=1,\ldots ,r\right\} .  
$$

With the notation above, 
consider the following conditions: 

\medskip 

\nd (K1) $d_1<d_2<d_3$ and $d_1+d_2+d_3>|\w |$. 

\smallskip 

\nd (K2) $d_2$ and $d_3$ do not belong to 
$\N d_1$ and $\Zn d_1+\Zn d_2$, respectively.

\smallskip

\nd (K3) $3d_2\neq 2d_3$ or 
$d_1+d_2<d_3+\Delta _{\w }(d_2,d_3)$. 

\smallskip 

\nd (K4) $sd_1\neq 2d_3$ 
for any odd number $s\geq 3$ 
or $d_1+d_2<d_3+\Delta _{\w }(d_1,d_3)$. 

\smallskip 

\nd (K5) $4d_1\neq 3d_2$ or 
$d_3<5\gcd (d_1,d_2)+\Delta _{\w }(2d_1,d_2)$. 

\medskip

Then, we have the following theorem, 
which will be proved in Section~\ref{sect:SUtheory} 
by means of the generalized Shestakov-Umirbaev theory.

\begin{theorem}\label{prop:main}
If $d_1,d_2,d_3\in \Gamma _+$ 
and $\w \in (\Gamma _+)^3$ satisfy {\rm (K1)} through {\rm (K4)}, 
then the following assertions hold$:$ 

\noindent{\rm (i)} 
Let $F\in \Aut _k\kx $ 
be such that $\mdegw F=(d_1,d_2,d_3)$. 
If $d_1$ and $d_2$ are linearly dependent over $\Z $, 
and $d_1$, $d_2$ and $d_3$ satisfy {\rm (K5)} and 
\begin{equation}\label{eq:prop:main}
d_1+d_2+d_3<\lcm (d_1,d_2)+\degw df_1\wedge df_2, 
\end{equation}
then $F$ does not belong to $\T_3(k)$.

\noindent{\rm (ii)} 
If $d_1$ and $d_2$ are linearly independent over $\Z $, 
then $(d_1,d_2,d_3)$ does not belong to 
$\mdegw \T_3(k)$. 
\end{theorem}

As a consequence of 
Theorem~\ref{prop:main} (i), 
we obtain the following theorem 
immediately.

\begin{theorem}\label{thm:cor}
Assume that 
$d_1,d_2,d_3\in \Gamma _+$ 
and $\w \in (\Gamma _+)^3$ satisfy 
{\rm (K1)} through {\rm (K5)}. 
If $d_1$ and $d_2$ are linearly dependent over $\Z $, 
and 
$$
d_1+d_2+d_3<\lcm (d_1,d_2)+\Delta _{\w }(d_1,d_2), 
$$
then $(d_1,d_2,d_3)$ does not belong to $\mdegw \T _3(k)$. 
\end{theorem}

Actually, 
if there exists $F\in \T _3(k)$ 
such that $\mdegw F=(d_1,d_2,d_3)$ 
under the assumption of Theorem~\ref{thm:cor}, 
then $F$ satisfies (\ref{eq:prop:main}), 
since 
$$
d_1+d_2+d_3<\lcm (d_1,d_2)+\Delta _{\w }(d_1,d_2)
\leq \lcm (d_1,d_2)+\degw df_1\wedge df_2. 
$$
Hence, 
$F$ does not belong to $\T _3(k)$ 
by Theorem~\ref{prop:main} (i), 
a contradiction.

Theorem~\ref{prop:main} (ii) 
implies the following useful criterion 
used in \cite{wild3} 
to prove many interesting theorems.

\begin{theorem}[cf.~{\cite[Section 5]{wild3}}]
\label{thm:indep degree}
Let $\w \in (\Gamma _+)^3$ 
be such that $w_1$, $w_2$ and $w_3$ 
are linearly independent over $\Z $. 
If $d_1,d_2,d_3\in \Gamma _+$ 
satisfy the following two conditions, 
then $(d_1,d_2,d_3)$ 
does not belong to $\mdegw \T _3(k)$$:$

\noindent{\rm (1)} 
$d_1$, $d_2$ and $d_3$ 
are linearly dependent over $\Z $, 
and are pairwise linearly independent over $\Z $. 

\noindent{\rm (2)} 
$d_i$ does not belong to $\Zn d_j+\Zn d_l$ 
for $(i,j,l)=(1,2,3),(2,3,1),(3,1,2)$. 
\end{theorem}

To check this theorem, 
we may assume that 
$(d_1,d_2,d_3)=\degw F$ for some $F\in \Aut _k\kx $. 
Then, 
$f_1^{\w }$, $f_2^{\w }$ and $f_3^{\w }$ are monomials, 
since $w_1$, $w_2$ and $w_3$ 
are linearly independent over $\Z $. 
Moreover, 
$d_1$, $d_2$ and $d_3$ are linearly dependent 
over $\Z $ if and only if 
$f_1^{\w }$, $f_2^{\w }$ and $f_3^{\w }$ 
are algebraically dependent over $k$, 
and hence if and only if 
$d_1+d_2+d_3=\degw F>|\w |$ 
by (\ref{eq:minimal autom}) 
and the preceding discussion. 
Using this remark, 
we can easily check that 
$d_{\sigma (1)}$, $d_{\sigma (2)}$ 
and $d_{\sigma (3)}$ satisfy 
(K1) through (K4) 
for some $\sigma \in \mathfrak{S}_3$ 
if (1) and (2) are satisfied.

Since $\Delta _{\w }(d,e)$'s 
are not easy to estimate in general, 
Theorems~\ref{prop:main} and \ref{thm:cor} 
are not suitable for practical use. 
So we derive a theorem which is useful in applications. 
Recall that $\Zn e_1+\cdots +\Zn e_r$ 
is a well-ordered subset of $\Gamma $ 
for each $e_1,\ldots ,e_r\in \Gamma _+$ 
with $r\in \N $ (cf.~\cite[Lemma 6.1 (ii)]{tame3}). 
Hence, 
we can define 
$$
|\w |_*:=\min (\{ 
\gamma \in \N w_{\sigma (1)}+\N w_{\sigma (2)}\mid 
\gamma >w_{\sigma (1)}+w_{\sigma (3)}\} 
\cup \{ 2w_{\sigma (1)}+w_{\sigma (3)}\} ), 
$$
where $\sigma \in \mathfrak{S}_3$ is such that 
$w_{\sigma (1)}\leq w_{\sigma (2)}\leq w_{\sigma (3)}$.

Consider the following conditions:

\medskip

\nd (A) One of the following holds:

\smallskip

\nd (A1) 
$3d_2\neq 2d_3$ and $sd_1\neq 2d_3$ 
for any odd number $s\geq 3$.

\smallskip

\nd (A2) $3d_2=2d_3$ and 
$d_1+d_2<d_3+\max \{ \Delta _{\w }(d_2,d_3),|\w |_*\} $.

\smallskip

\nd (A3) $sd_1=2d_3$ for some odd number $s\geq 3$ and 
$$
d_1+d_2<d_3+\max \{ \Delta _{\w }(d_1,d_3),|\w |_*\} . 
$$

\medskip

\nd (B) 
$d_1$ and $d_2$ are linearly independent over $\Z $, 
or $d_1$ and $d_2$ are linearly dependent over $\Z $ 
and one of the following holds:

\smallskip

\nd (B1) $\gcd (d_1,d_2)\leq |\w |_*$ 
and $d_3$ belongs to $\N \gcd (d_1,d_2)$.

\smallskip

\nd (B2) $d_1+d_2+d_3<\lcm (d_1,d_2)+|\w |_*$.

\medskip

The following is our main theorem.

\begin{theorem}\label{thm:main}
If $d_1,d_2,d_3\in \Gamma _+$ and 
$\w \in (\Gamma _+)^3$ 
satisfy 
{\rm (K1)}, {\rm (K2)}, {\rm (A)} and {\rm (B)}, 
then $(d_1,d_2,d_3)$ does not belong to $\mdegw \T_3(k)$. 
\end{theorem}

We prove this theorem in Section~\ref{sect:pf main} 
using Theorem~\ref{prop:main} and a weighted version of 
Sun-Chen's lemma (Theorem~\ref{thm:|w|*}).

To end this section, 
we show Theorem~\ref{thm:total degree} 
by assuming Theorem~\ref{thm:main}. 
Take any $d_1,d_2,d_3\in \N $ with 
$d_1\leq d_2\leq d_3$ which satisfy 
(a), (b) and (c). 
It suffices to check that 
(K1), (K2), (A) and (B) hold 
for $\Gamma =\Z $ and $\w =(1,1,1)$. 
Since $d_1\leq d_2\leq d_3$, 
we have $d_1<d_2<d_3$ by (c). 
Since $d_1\geq 1$, 
it follows that $d_1+d_2+d_3>3=|\w |$, 
proving (K1). 
Clearly, 
(c) implies (K2). 
If (a1) holds, 
then we have (A1). 
Assume that (a1) does not hold. 
Then, 
we have $3d_2=2d_3$ or 
$sd_1=2d_3$ for some odd number $s\geq 3$, 
and (a2) is satisfied. 
Since $|\w |_*=3$, 
it follows that 
$$
d_1+d_2\leq d_3+2<d_3+|\w |_*. 
$$
Thus, we get (A2) or (A3). 
Since $|\w |_*=3$, 
(b1) and (b2) imply (B1) and (B2), 
respectively. 
Therefore, 
$d_1$, $d_2$ and $d_3$ satisfy 
(K1), (K2), (A) and (B).

\section{Generalized Shestakov-Umirbaev theory}
\label{sect:SUtheory}
\setcounter{equation}{0}

The goal of this section is to prove 
Theorem~\ref{prop:main}. 
First, 
we briefly recall the generalized 
Shestakov-Umirbaev theory 
\cite{SUineq}, \cite{tame3}. 
For the moment, 
$n\in \N $ may be arbitrary. 
Let $\w $ be an element of $(\Gamma _+)^n$, 
and let $F=(f_1,f_2,f_3)$ and $G=(g_1,g_2,g_3)$ 
be triples of elements of $\kx $ 
such that $f_1$, $f_2$, $f_3$ and $g_1$, $g_2$, $g_3$ 
are algebraically independent over $k$, 
respectively. 
We say that the pair $(F,G)$ satisfies the 
{\it Shestakov-Umirbaev condition} for the weight $\w $ 
if the following conditions hold (cf.~\cite{tame3}): 

\medskip

\noindent
(SU1) $g_1=f_1+af_3^2+cf_3$ and $g_2=f_2+bf_3$ 
for some $a,b,c\in k$, 
and $g_3-f_3$ belongs to $k[g_1,g_2]$; 

\smallskip 

\noindent
(SU2) $\deg _{\w }f_1\leq \deg _{\w }g_1$ 
and $\deg _{\w }f_2=\deg _{\w }g_2$; 

\smallskip 

\noindent
(SU3) $(g_1^{\w })^2\approx (g_2^{\w })^s$ 
for some odd number $s\geq 3$; 

\smallskip 

\noindent
(SU4) $\deg _{\w }f_3\leq \deg _{\w }g_1$, 
and $f_3^{\w }$ does not belong to 
$k[g_1^{\w }, g_2^{\w }]$; 

\smallskip

\noindent
(SU5) 
$\deg _{\w }g_3<\deg _{\w }f_3$; 

\smallskip 

\noindent
(SU6) 
$\deg _{\w }g_3<\deg _{\w }g_1-\deg _{\w }g_2 
+\deg _{\w }dg_1\wedge dg_2$. 

\medskip

Here, 
$h_1\approx h_2$ 
(resp.\ $h_1\not\approx h_2$) denotes that 
$h_1$ and $h_2$ are linearly dependent 
(resp.\ linearly independent) over $k$ 
for each $h_1,h_2\in \kx \sm \zs $. 
We say that $(F,G)$ 
satisfies the {\it weak Shestakov-Umirbaev condition} 
for the weight $\w $ if (SU4), (SU5), (SU6) and 
the following conditions are satisfied (cf.~\cite{tame3}):

\medskip 

\noindent
(SU$1'$) $g_1-f_1$, 
$g_2-f_2$ and $g_3-f_3$ belong to $k[f_2,f_3]$, 
$k[f_3]$ and $k[g_1,g_2]$, respectively; 

\smallskip 

\noindent
(SU$2'$) $\deg f_i\leq \deg g_i$ for $i=1,2$; 

\smallskip 

\noindent
(SU$3'$) $\deg g_2<\deg g_1$, and 
$g_1^{\w }$ does not belong to $k[g_2^{\w }]$. 

\medskip 

\noindent
It is easy to check that 
(SU1), (SU2) and (SU3) imply 
(SU$1'$), (SU$2'$) and (SU$3'$), 
respectively. 
Hence, 
the Shestakov-Umirbaev condition 
implies the weak Shestakov-Umirbaev condition. 
If $(F,G)$ satisfies the weak Shestakov-Umirbaev condition 
for the weight $\w $, 
then $(F,G)$ has the following properties 
due to~\cite[Theorem 4.2]{tame3}. 
Here, 
we regard $\Gamma $ as a subgroup of $\Q \otimes _{\Z }\Gamma $ 
which has a structure of totally ordered additive group 
induced from $\Gamma $. 
We note that $\delta $ in (P1) 
is equal to $\gcd (\degw g_1,\degw g_2)$ in our notation.

\medskip 

\noindent
{\rm (P1)} $(g_1^{\w })^2\approx (g_2^{\w })^s$ 
for some odd number $s\geq 3$, 
and so $\delta :=(1/2)\degw g_2$ belongs to $\Gamma $. 

\smallskip 

\noindent
{\rm (P2)} $\degw f_3\geq (s-2)\delta +\degw dg_1\wedge dg_2$. 

\smallskip 

\noindent
{\rm (P5)} If $\degw f_1<\degw g_1$, 
then $s=3$, $g_1^{\w }\approx (f_3^{\w })^2$, 
$\degw f_3=(3/2)\delta $ and 
$$
\degw f_1\geq \frac{5}{2}\delta +\degw dg_1\wedge dg_2. 
$$

\smallskip 

\noindent
{\rm (P7)} $\degw f_2<\degw f_1$, $\degw f_3\leq \degw f_1$, 
and $\delta <\degw f_i\leq s\delta $ for $i=1,2,3$. 

\medskip

\noindent
{\rm (P11)} If $\degw f_3\leq \degw f_2$, 
then $s=3$. 

\medskip

We say that $F\in \Aut _k\kx $ 
admits an {\it elementary reduction} 
for the weight $\w $ 
if there exists an elementary automorphism $E$ 
such that $\degw F\circ E<\degw F$. 
For a $k$-vector subspace $V$ of $\kx $, 
we define $V^{\w }$ to be the $k$-vector subspace 
of $\kx $ generated by $\{ f^{\w }\mid f\in V\} $. 
Then, 
$F\in \Aut _k\kx $ admits an elementary reduction 
for the weight $\w $ 
if and only if 
$f_l^{\w }$ belongs to $k[\{ f_i\mid i\neq l\} ]^{\w }$ 
for some $1\leq l\leq n$.

Now, 
assume that $n=3$. 
For $\sigma \in \mathfrak{S}_3$ 
and $F\in \Aut _k\kx $, 
we define 
$F_{\sigma }=
(f_{\sigma (1)},f_{\sigma (2)},f_{\sigma (3)})$. 
We say that $F$ admits a 
{\it Shestakov-Umirbaev reduction} 
for the weight $\w$ 
if there exist $G\in \Aut _k\kx $ 
and $\sigma \in \mathfrak{S}_3$ 
such that $(F_{\sigma },G_{\sigma })$ 
satisfies the Shestakov-Umirbaev condition 
for the weight $\w $.

The following theorem is a generalization of 
the main result of Shestakov-Umirbaev~\cite{SU}.

\begin{theorem}[{\cite[Theorem 2.1]{tame3}}]\label{thm:SU}
If $\degw F>|\w |$ holds for 
$F\in \T_3(k)$ and $\w \in (\Gamma_+)^3$, 
then $F$ admits an elementary reduction 
or a Shestakov-Umirbaev reduction for the weight $\w $. 
\end{theorem}

We also use the following 
version of the Shestakov-Umirbaev inequality 
(see~\cite[Section 3]{tame3} for detail). 
Let $S=\{ f,g\} $ 
be a subset of $\kx $ 
such that $f$ and $g$ are 
algebraically independent over $k$, 
and $h$ a nonzero element of $k[S]$. 
Then, we can uniquely express 
$h=\sum _{i,j}c_{i,j}f^ig^j$, 
where $c_{i,j}\in k$ for each $i,j\in \Zn $. 
We define $\degw ^Sh$ to be the maximum among 
$\degw f^ig^j$ for $i,j\in \Zn $ with $c_{i,j}\neq 0$. 
We note that, 
if $h^{\w }$ does not belong to $k[f^{\w },g^{\w }]$, 
or if $\degw h<\degw g$ and $h^{\w }$ 
does not belong to $k[f^{\w }]$, 
then $\degw h<\degw ^Sh$.

With the notation and assumption above, 
the following lemma holds 
(see \cite[Lemmas 3.2 (i) and 3.3 (ii)]{tame3} 
for the proof).

\begin{lemma}\label{lem:SU ineq}
If $\degw h<\degw ^Sh$, 
then there exist $p,q\in \N $ with $\gcd (p,q)=1$ 
such that $(g^{\w })^p\approx (f^{\w })^q$ and  
$$
\degw h\geq q\degw f-\degw f-\degw g+\degw df\wedge dg. 
$$ 

Assume further that $\degw f<\degw g$ 
and $g^{\w }$ does not belong to $k[f^{\w }]$. 
If $\degw h\leq \degw g$, 
then we have $p=2$, 
and $q$ is an odd number with $q\geq 3$. 
If $\degw h\leq \degw f$, 
then we have $q=3$. 
\end{lemma}

Now, 
we prove Theorem~\ref{prop:main}. 
First, 
consider the following conditions 
for $e_1,e_2,e_3\in \Gamma _+$ 
and $\w \in (\Gamma _+)^3$: 

\smallskip 

\nd (K2$'$) 
$e_1$ and $e_2$ do not belong to 
$\Zn e _2+\Zn e_3$ and $\N e_3$, 
respectively. 

\smallskip

\nd (K3$'$) $2e_1=3e_2$ and 
$e_2+e_3\geq e_1+\Delta _{\w }(e_1,e_2)$. 

\smallskip 

\nd (K4$'$) $2e_1=se_3$ 
for some odd number $s\geq 3$ 
and $e_2+e_3\geq e_1+\Delta _{\w }(e_1,e_3)$. 

\smallskip 

\nd (K5$'$) $3e_2=4e_3$ 
and $e_1\geq 5\gcd (e_2,e_3)+\Delta _{\w }(e_2,2e_3)$.

\smallskip

The following lemma is a direct 
consequence of the generalized Shestakov-Umirbaev theory.

\begin{lemma}\label{lem:key}
Let $e_1,e_2,e_3\in \Gamma _+$ 
and $\w \in (\Gamma _+)^3$ 
be such that $e_1>e_2>e_3$ 
and $\mdegw F=(e_1,e_2,e_3)$ for some $F\in \T_3(k)$.

\noindent{\rm (i)} 
If $F$ admits a Shestakov-Umirbaev reduction 
for the weight $\w $, 
then one of 
{\rm (K3$'$)}, {\rm (K4$'$)} and {\rm (K5$'$)} holds. 

\noindent{\rm (ii)} 
Assume that {\rm (K2$'$)} is satisfied. 
Then, 
the following assertions hold$:$ 

\nd {\rm (1)} 
If $f_3^{\w }$ belongs to $k[f_1,f_2]^{\w }$, 
then we have {\rm (K3$'$)}.

\nd {\rm (2)} 
If $f_2^{\w }$ belongs to $k[f_1,f_3]^{\w }$, 
then we have {\rm (K4$'$)}. 

\nd {\rm (3)} 
If $f_1^{\w }$ belongs to $k[f_2,f_3]^{\w }$, 
then $e_2$ and $e_3$ are linearly dependent over $\Z $ 
and 
\begin{equation}\label{eq:lem:c}
e_1+e_2+e_3\geq \lcm (e_2,e_3)+\degw df_2\wedge df_3. 
\end{equation}
\end{lemma}
\begin{proof}
(i) 
By assumption, 
there exist $G\in \Aut _k\kx $ 
and $\sigma \in \mathfrak{S}_3$ 
such that $(F_{\sigma },G_{\sigma })$ 
satisfies the Shestakov-Umirbaev condition 
for the weight $\w $. 
Then, 
we see from (SU1) that $G$ is tame, 
since so is $F$. 
Since $e_1>e_i$ for $i=2,3$, 
we have $\sigma (1)=1$ in view of (P7). 
Hence, 
$\sigma $ must be the identity permutation 
$\id $ or the transposition $(2,3)$. 
First, 
assume that $\sigma =\id $. 
Then, 
we have $s=3$ by (P11), 
since $e_2>e_3$. 
By (SU2), we get $\degw g_1\leq e_1$ 
and $\degw g_2=e_2$. 
When $\degw g_1=e_1$, 
(SU3) implies 
$$
2e_1
=2\degw g_1=s\degw g_2=3\degw g_2=3e_2. 
$$
Since $G$ is tame, 
and $\degw g_i=e_i$ for $i=1,2$, 
we know by (P2) that 
$$
e_3
\geq (3-2)\frac{1}{2}e_2+\degw dg_1\wedge dg_2
\geq e_1-e_2+\Delta _{\w }(e_1,e_2). 
$$ 
Therefore, 
{\rm (K3$'$)} holds in this case. 
If $\degw g_1<e_1$, 
then we have $\degw g_1=2e_3$ 
and $e_3=(3/2)(1/2)e_2$ by (P5). 
Hence, 
we get $3e_2=4e_3$, 
$\gcd (e_2,e_3)=(1/4)e_2$ 
and $\degw dg_1\wedge dg_2\geq \Delta _{\w }(2e_3,e_2)$. 
Thus, 
the last part of (P5) implies 
$$
e_1\geq \frac{5}{2}\frac{1}{2}e_2
+\degw dg_1\wedge dg_2
\geq 5\gcd (e_2,e_3)+\Delta _{\w }(2e_3,e_2). 
$$
Therefore, 
{\rm (K5$'$)} holds in this case. 
Next, 
assume that $\sigma =(2,3)$. 
Then, 
we have $\degw g_3=\degw f_3=e_3$ 
by (SU2). 
We also have 
$\degw g_1=\degw f_1=e_1$, 
for otherwise $e_2=(3/2)(1/2)e_3$ by (P5), 
and so $e_2<e_3$, 
a contradiction. 
Hence, 
it follows from (SU3) that $2e_1=se_3$ 
for some odd number $s\geq 3$. 
Thus, we get 
$$
e_2\geq (s-2)\frac{1}{2}e_3+\degw dg_1\wedge dg_3
\geq e_1-e_3+\Delta _{\w }(e_1,e_3) 
$$
by (P2). 
Therefore, 
{\rm (K4$'$)} holds in this case.

(ii) 
In the case of (1), 
there exists $h \in k[f_1,f_2]$ 
such that $h ^{\w }=f_3^{\w }$. 
Then, 
we have $\degw h =e_3<e_i$ for $i=1,2$. 
Since $h^{\w }=f_3^{\w }$ is not a constant, 
it follows that 
$h ^{\w }$ does not belong to 
$k[f_1^{\w },f_2^{\w }]$. 
This implies that 
$\degw h <\degw ^Sh $ holds for $S=\{ f_1,f_2\} $ 
as mentioned. 
By Lemma~\ref{lem:SU ineq}, 
there exist 
$p,q\in \N $ with $\gcd (p,q)=1$ 
such that $(f_1^{\w })^p\approx (f_2^{\w })^q$ and 
\begin{equation}\label{eq:keylem:1}
\begin{aligned}
e_3=\degw h 
&\geq 
q\degw f_2-\degw f_1-\degw f_2+\degw df_1\wedge df_2 \\
&\geq qe_2-e_1-e_2+\Delta _{\w }(e_1,e_2). 
\end{aligned}
\end{equation}
We note that $\degw h<\degw f_2<\degw f_1$, 
and $f_1^{\w }$ does not belong to 
$k[f_2^{\w }]$ by (K2$'$). 
Hence, 
we have $(p,q)=(2,3)$ 
by the last part of Lemma~\ref{lem:SU ineq}. 
Since $(f_1^{\w })^p\approx (f_2^{\w })^q$, 
it follows that $2e_1=3e_2$, 
and thus $qe_2-e_1-e_2=e_1-e_2$. 
Therefore, 
we obtain {\rm (K3$'$)} from (\ref{eq:keylem:1}).

We prove (2) similarly. 
Let $h \in k[f_1,f_3]$ be 
such that $h ^{\w }=f_2^{\w }$. 
Then, 
we have $\degw h =e_2<e_1=\degw f_1$. 
Since $e_2$ does not belong to $\N e_3$ by (K2$'$), 
$h ^{\w }$ does not belong to $k[f_3^{\w }]$. 
Hence, 
$\degw h <\degw ^Sh $ holds 
for $S=\{ f_1,f_3\} $ as mentioned. 
As before, 
there exist 
$p,q\in \N $ with $\gcd (p,q)=1$ 
such that $pe_1=qe_3$ and 
\begin{equation}\label{eq:keylem:2}
e_2=\degw h 
\geq qe_3-e_1-e_3+\Delta _{\w }(e_1,e_3) 
\end{equation}
by Lemma~\ref{lem:SU ineq}. 
We note that 
$\degw f_3$ and $\degw h $ 
are less than $\degw f_1$, 
and $f_1^{\w }$ does not belong to $k[f_3^{\w }]$ 
by (K2$'$). 
Hence, 
we know that $p=2$, 
and $q$ is an odd number with $q\geq 3$ 
by the last part of Lemma~\ref{lem:SU ineq}. 
Thus, 
we get $2e_1=qe_3$, 
and so $qe_3-e_1-e_3=e_1-e_3$. 
Therefore, 
we obtain {\rm (K4$'$)} from (\ref{eq:keylem:2}).

Finally, 
we prove (3). 
Let $h \in k[f_2,f_3]$ be 
such that $h ^{\w }=f_1^{\w }$. 
Then, 
we have $\degw h =e_1$. 
Since $e_1$ does not belong to 
$\Zn e_2+\Zn e_3$ by (K2$'$), 
it follows that 
$h ^{\w }$ does not belong to 
$k[f_2^{\w },f_3^{\w }]$. 
This implies that 
$\degw h <\degw ^Sh $ holds for $S=\{ f_2,f_3\} $. 
As before, 
there exist 
$p,q\in \N $ with $\gcd (p,q)=1$ 
such that $pe_3=qe_2$ and 
$
e_1=\degw h \geq 
qe_2-e_2-e_3+\degw df_2\wedge df_3. 
$
Since $qe_2=\lcm (e_2,e_3)$, 
this inequality yields (\ref{eq:lem:c}), 
proving (3). 
\end{proof}

Now, 
let us complete the proof of Theorem~\ref{prop:main}. 
Set $e_1:=d_3$, $e_2:=d_2$ and $e_3:=d_1$. 
Then, 
we have $e_1>e_2>e_3$ by (K1). 
First, 
we prove (i) by contradiction. 
Suppose that $F$ is tame. 
Then, 
$F':=(f_3,f_2,f_1)$ is also tame. 
Since $\degw F'>|\w |$ by (K1), 
it follows that 
$F'$ admits a Shestakov-Umirbaev reduction 
or an elementary reduction for the weight $\w $ 
by Theorem~\ref{thm:SU}. 
Since $\mdegw F'=(e_1,e_2,e_3)$, 
we know by Lemma~\ref{lem:key} that 
one of (K3$'$), (K4$'$), (K5$'$) 
and (\ref{eq:lem:c}) must be satisfied. 
This contradicts 
(K3), (K4), (K5) and (\ref{eq:prop:main}). 
Therefore, 
$F$ does not belong to $\T_3(k)$, 
proving (i).

We prove (ii) similarly. 
Suppose that $(d_1,d_2,d_3)$ 
belongs to $\mdegw \T_3(k)$. 
Then, $(e_1,e_2,e_3)$ also 
belongs to $\mdegw \T_3(k)$. 
Hence, 
there exists $F'\in \T_3(k)$ 
such that $\mdegw F'=(e_1,e_2,e_3)$. 
As in the case of (i), 
this implies that 
one of (K3$'$), (K4$'$) and (K5$'$) holds, 
or $e_2$ and $e_3$ are linearly dependent over $\Z $ 
by Lemma~\ref{lem:key}. 
However, 
(K3$'$) and (K4$'$) do not hold 
because of (K3) and (K4). 
Since $e_2=d_2$ and $e_3=d_1$ 
are linearly independent over $\Z $ by assumption, 
the rest of the conditions also do not hold. 
This is a contradiction. 
Therefore, 
$(d_1,d_2,d_3)$ 
does not belong to $\mdegw \T_3(k)$, 
proving (ii). 
This completes the proof of Theorem~\ref{prop:main}.

\section{A generalization of Sun-Chen's lemma}\label{sect:w}
\setcounter{equation}{0}

In this section, 
we discuss lower bounds for 
the degrees of differential forms. 
Note that, 
if $df_1\wedge \cdots \wedge df_r\neq 0$ 
for $f_1,\ldots ,f_r\in \kx $ with $r\geq 1$,
then we have 
$\deg df_1\wedge \cdots \wedge df_r\geq r$ by definition. 
In the case of $r=2$, 
Kara\'s~\cite[Theorem 14]{456} showed that 
$\deg df_1\wedge df_2\geq 4$ 
under the additional assumption that 
$\deg f_1=4$ and $\deg f_2=6$. 
When $n=3$, 
Sun-Chen~\cite[Lemma 3.1]{SunChen} 
showed that 
$\deg df_1\wedge df_2\geq 3$ 
for each $F\in \Aut _k\kx $ with 
$\deg f_1\nmid \deg f_2$ and 
$\deg f_2\nmid \deg f_1$.

The purpose of this section is to prove 
the following theorem. 
This theorem is considered as a generalization of 
Sun-Chen's lemma, 
since $|\w |_*=3$ if $\Gamma =\Z $ and $\w =(1,1,1)$.

\begin{theorem}\label{thm:|w|*}
Assume that $n=3$. 
Let $F\in \Aut _k\kx $ and $\w \in (\Gamma _+)^3$ 
be such that 
$\degw f_1$ or $\degw f_2$ 
does not belong to 
$\{ w_1,w_2,w_3\} $, 
and $\degw f_i$ does not belong to $\N \degw f_j$ 
for each $(i,j)\in \{ (1,2),(2,1)\} $. 
Then, 
we have $\degw df_1\wedge df_2\geq |\w |_*$. 
\end{theorem}

As a consequence of Theorem~\ref{thm:|w|*}, 
we know that 
$\Delta _{\w }(d_1,d_2)\geq |\w |_*$ 
holds for $d_1,d_2\in \Gamma _+$ 
and $\w \in (\Gamma _+)^3$ 
if $d_1$ or $d_2$ does not belong to 
$\{ w_1,w_2,w_3\} $, 
and $d_i$ does not belong to $\N d_j$ 
for each $(i,j)\in \{ (1,2),(2,1)\} $. 
We also note that 
Kara\'s~\cite[Theorem 14]{456} implies 
$\Delta _{\w }(4,6)\geq 4$ 
when $\Gamma =\Z $ and $\w =(1,1,1)$.

To prove Theorem~\ref{thm:|w|*}, 
we need the following two propositions.

\begin{proposition}\label{prop:vdk}
Assume that $n=2$. 
Let $F\in \Aut _k\kx $ 
and $\w \in (\Gamma _+)^2$ 
be such that $\degw F>|\w |$. 
Then, 
there exist $(i,j)\in \{ (1,2),(2,1)\} $ 
and $l\in \N $ such that 
$f_i^{\w }\approx (f_j^{\w })^l$. 
Hence, 
$\degw f_i$ belongs to $\N \degw f_j$. 
\end{proposition}
\begin{proof}
Since $\degw F>|\w |$ by assumption, 
$f_1^{\w }$ and $f_2^{\w }$ 
are algebraically dependent over $k$ 
by (\ref{eq:minimal autom}) and the preceding discussion. 
Hence, 
$x_l$ does not belong to 
$k[f_1^{\w },f_2^{\w }]$ 
for some $l\in \{ 1,2\} $. 
On the other hand, 
$x_l$ belongs to $\kx =k[f_1,f_2]$. 
Thus, 
$\degw ^Sx_l>\degw x_l$ holds for $S=\{ f_1,f_2\} $. 
By Lemma~\ref{lem:SU ineq}, 
there exist $p_1,p_2\in \N $ with $\gcd (p_1,p_2)=1$ 
such that $(f_2^{\w })^{p_1}\approx (f_1^{\w })^{p_2}$ 
and  
$$
w_l=
\degw x_l\geq 
p_2d_1-d_1-d_2+\degw df_1\wedge df_2, 
$$ 
where $d_i:=\degw f_i$ for $i=1,2$. 
Then, 
we have $d_i=p_i\gcd (d_1,d_2)$ for $i=1,2$. 
Since $\degw df_1\wedge df_2=|\w |$, 
it follows that 
$$
(p_1+p_2-p_1p_2)\gcd (d_1,d_2)
=d_1+d_2-p_2d_1
\geq |\w |-w_l\geq \min \{ w_1,w_2\} >0. 
$$
Hence, we get $p_1p_2<p_1+p_2$. 
Since $\gcd (p_1,p_2)=1$, 
this implies that $p_1=1$ or $p_2=1$. 
Therefore, we have 
$f_2^{\w }\approx (f_1^{\w })^{p_2}$ 
or $(f_2^{\w })^{p_1}\approx f_1^{\w }$. 
\end{proof}

Next, 
assume that $n\geq 3$, 
and let $\w \in (\Gamma _+)^n$ 
be such that $w_1\leq \cdots \leq w_n$. 
We define 
$$
|\w |_*:=\min (\{ 
\gamma \in \N w_1+\cdots +\N w_{n-1}\mid 
\gamma >|\w |-w_{n-1}\} \cup \{ |\w |+w_1-w_{n-1}\} ). 
$$
Note that, 
if $n=3$, 
then this is the same as $|\w |_*$ defined before.

With this notation, 
we have the following proposition.

\begin{proposition}\label{prop:key *}
If $f_1,\ldots ,f_{n-1}\in \kx $ 
are algebraically independent over $k$ 
and $\degw df_1\wedge \cdots \wedge df_{n-1}<|\w |_*$, 
then there exist $\phi \in \T_n(k)$ and $1\leq l\leq n$ 
such that 
$\phi (f_i)$ belongs to 
$k[\x \sm \{ x_l\} ]$ 
and $\degw \phi (f_i)=\degw f_i$ 
for $i=1,\ldots ,n-1$. 
\end{proposition}

Before proving Proposition~\ref{prop:key *}, 
we show how Propositions~\ref{prop:vdk} 
and~\ref{prop:key *} imply Theorem~\ref{thm:|w|*}. 
By changing the indices of $x_1$, $x_2$ and $x_3$, 
and $w_1$, $w_2$ and $w_3$ if necessary, 
we may assume that $w_1\leq w_2\leq w_3$. 
Suppose to the contrary that 
$\degw df_1\wedge df_2<|\w |_*$ 
for some $F\in \Aut _k\kx $ and $\w \in (\Gamma _+)^3$ 
which satisfy the assumptions of Theorem~\ref{thm:|w|*}. 
By Proposition~\ref{prop:key *}, 
there exist $\phi \in \T_3(k)$ and $1\leq i_1<i_2\leq 3$ 
such that $\phi (f_i)$ belongs to $k[x_{i_1},x_{i_2}]$ 
and $\degw \phi (f_i)=\degw f_i$ for $i=1,2$. 
Then, 
$G:=(\phi (f_1),\phi (f_2))$ 
belongs to $\Aut _kk[x_{i_1},x_{i_2}]$. 
Set $\vv :=(w_{i_1},w_{i_2})$. 
Then, 
we have 
$$
d_i:=\degw f_i=\degw \phi (f_i)=\degv \phi (f_i)
$$
for $i=1,2$. 
If $\degv G=|\vv |$, 
then $(d_1,d_2)=\mdegv G$ 
is equal to $(w _{i_1},w_{i_2})$ or $(w _{i_2},w_{i_1})$ 
by the remark after Lemma~\ref{prop:deg F}. 
Hence, 
$d_1$ and $d_2$ both belong to $\{ w_1,w_2,w_3\} $, 
a contradiction. 
If $\degv G>|\vv |$, 
then $d_i$ belongs to $\N d_j$ 
for some $(i,j)\in \{ (1,2),(2,1)\} $ 
by Proposition~\ref{prop:vdk}, 
a contradiction. 
Therefore, 
Theorem~\ref{thm:|w|*} holds true.

To prove Proposition~\ref{prop:key *}, 
we need some lemmas. 
Let $\kxr $ be the field of fractions of $\kx $. 
Then, 
the following assertion holds.

\begin{lemma}\label{lem:tbasis}
Assume that $g_1,\ldots ,g_r\in \kxr $ 
and $1\leq l\leq n$ 
satisfy the following conditions, 
where $1\leq r<n$. 
Then, 
$g_1,\ldots ,g_r$ belong to $k(\x \sm \{ x_l\} )$. 

\nd {\rm (1)} 
$g_1,\ldots ,g_r$ are algebraically independent over $k$.

\nd {\rm (2)} 
$g_1,\ldots ,g_r$, $x_{i_1},\ldots ,x_{i_{n-r}}$ 
are algebraically dependent over $k$ for any distinct 	
$i_1,\ldots ,i_{n-r}\in \{ 1,\ldots ,n\} \sm \{ l\} $. 
\end{lemma}
\begin{proof}
Set 
$S=\{ x_i\mid i\neq l\} \cup \{ g_1,\ldots ,g_r\} $. 
If the assertion is false, 
then $\kxr $ is algebraic over $k(S)$. 
By (1), 
there exists a transcendence basis $S_0$ of $\kxr $ 
over $k$ such that 
$\{ g_1,\ldots ,g_r\} \subset S_0\subset S$. 
Then, 
we have 
$$
S_0=\{ g_1,\ldots ,g_r,x_{i_1},\ldots ,x_{i_{n-r}}\} 
$$ 
for some distinct 
$i_1,\ldots ,i_{n-r}\in \{ 1,\ldots ,n\} \sm \{ l\} $. 
This contradicts (2). 
\end{proof}

Now, 
write 
$$
\omega :=df_1\wedge \cdots \wedge df_{n-1}
=\sum _{i=1}^nh_i\eta _i,
$$
where $h_i\in \kx $ and 
$\eta _i:=dx_1\wedge \cdots \wedge 
dx_{i-1}\wedge dx_{i+1}\wedge \cdots \wedge dx_n$
for each $i$. 
Then, 
we have $\degw \omega \geq \degw h_i\eta _i$ for $i=1,\ldots ,n$. 
For each $\phi \in \Aut _k\kx $, 
it holds that 
\begin{equation}\label{eq:omega ^phi note}
\omega ^{\phi }
:=d\phi (f_1)\wedge \cdots \wedge d\phi (f_{n-1})
=\sum _{i=1}^n\phi (h_i)\eta _i^{\phi }, 
\end{equation}
where 
$\eta _i^{\phi }:=
d\phi (x_1)\wedge \cdots \wedge 
d\phi (x_{i-1})\wedge d\phi (x_{i+1})
\wedge \cdots \wedge d\phi (x_n)$ 
for each $i$. 
In fact, we have 
\begin{align*}
&d\phi (f_i)
=d(f_i(\phi (x_1),\ldots ,\phi (x_n))) \\
&\quad =\sum _{j=1}^n
\left(\frac{\partial f_i}{\partial x_j}\right) 
(\phi (x_1),\ldots ,\phi (x_n))
d\phi (x_j)
=\sum _{j=1}^n\phi \left( 
\frac{\partial f_i}{\partial x_j}
\right) d\phi (x_j)
\end{align*}
for each $i$ by chain rule. 
We note that, 
if $\phi $ is an elementary automorphism 
such that $\phi (x_l)=x_l+p$ 
for some $1\leq l\leq n$ and $p\in k[\x \sm \{ x_l\} ]$, 
then 
\begin{equation}\label{eq:diff elem}
\eta _l^{\phi }=\eta _l
\quad \text{and}\quad 
\eta _i^{\phi }
=\eta _i-(-1)^{i-l}\frac{\partial p}{\partial x_i}\eta _l
\quad \text{for each}\quad i\neq l. 
\end{equation}

With the notation above, 
we have the following lemma.

\begin{lemma}\label{lem:ind key}
If $\degw \omega <|\w |_*$, 
then there exist $\phi \in \T_n(k)$, 
$h\in \kx $ and $1\leq l\leq n$ 
such that 
$\omega ^{\phi }=h\eta _l$ 
and $\degw \phi (f_i)=\degw f_i$ for $i=1,\ldots ,n-1$. 
\end{lemma}

Before proving Lemma~\ref{lem:ind key}, 
we show Proposition~\ref{prop:key *} 
by assuming this lemma. 
Since $\degw \omega <|\w |_*$ 
by assumption, 
there exist $\phi \in \T_n(k)$, 
$h\in \kx $ and $1\leq l\leq n$ 
as in Lemma~\ref{lem:ind key}. 
We have only to check that 
$\phi (f_i)$ belongs to $k[\x \sm \{ x_l\} ]$ 
for $i=1,\ldots ,n-1$. 
Since 
$\omega ^{\phi }=h\eta _l$, 
we have 
$$
d\phi (f_1)\wedge \cdots \wedge 
d\phi (f_{n-1})\wedge dx_i
=\omega ^{\phi }\wedge dx_i
=h\eta _l\wedge dx_i=0
$$ 
for each $i\neq l$. 
Hence, 
$\phi (f_1),\ldots ,\phi (f_{n-1})$, $x_i$ 
are algebraically dependent over $k$ 
for each $i\neq l$. 
Thus, 
$\phi (f_1),\ldots ,\phi (f_{n-1})$ 
belong to $k[\x \sm \{ x_l\} ]$ 
by Lemma~\ref{lem:tbasis}. 
Therefore, 
Proposition~\ref{prop:key *} 
follows from Lemma~\ref{lem:ind key}.

To prove Lemma~\ref{lem:ind key}, 
we use the following two lemmas.

\begin{lemma}\label{lem:deg pres aut}
Assume that $\sigma \in \Aut _k\kx $ 
is such that $\sigma (x_i)=\alpha _ix_i+p_i$ 
for some $\alpha _i\in k^{\times }$ 
and $p_i\in k[x_1,\ldots ,x_{i-1}]$ 
with $\degw p_i\leq w_i$ 
for $i=1,\ldots ,n$. 
Then, 
$\degw \sigma (f)=f$ 
holds for each $f\in \kx $. 
\end{lemma}
\begin{proof}
We may assume that $f\neq 0$. 
Since $\degw p_i\leq w_i$ by assumption, 
we have $\degw \sigma (x_i)=w_i$ for each $i$. 
Hence, 
the assertion holds when $f$ is a monomial 
in view of (\ref{eq:deg prop}). 
This implies that 
$\degw \sigma (f)\leq \degw f$ 
in the general case. 
Since 
$\sigma (f)=\sigma (f^{\w })+\sigma (f-f^{\w })$ 
with 
$$
\degw \sigma (f-f^{\w })
\leq \degw {(f-f^{\w })}<\degw f, 
$$
it suffices to show that 
$\degw \sigma (f^{\w })=\degw f$. 
Since $\degw p_i\leq w_i$, 
each $\sigma (x_i)^{\w }$ 
has the form $\alpha _ix_i+p_i'$ 
for some $p_i'\in k[x_1,\ldots ,x_{i-1}]$. 
Define 
$\tau \in \Aut _k\kx $ by 
$\tau (x_i)=\sigma (x_i)^{\w }$ 
for $i=1,\ldots ,n$. 
Then, 
$\sigma (m)^{\w }=\tau (m)$ 
holds for each monomial $m$ by (\ref{eq:deg prop}). 
Put $\alpha =\degw f$ 
and write $f^{\w }=\sum _im_i$, 
where $m_i$'s are nonzero monomials of $\kx _{\alpha }$. 
Then, 
we have 
$\degw \sigma (m_i)=\degw m_i=\alpha $ 
for each $i$, 
and 
$\sum _i\sigma (m_i)^{\w }
=\sum _i\tau (m_i)
=\tau \left( 
\sum _im_i
\right) 
=\tau (f^{\w})\neq 0$. 
This implies that 
$$
\sum _i\sigma (m_i)^{\w }
=\left( \sum _i\sigma (m_i)\right) ^{\w }
=\sigma \left(\sum _im_i\right)^{\w }
=\sigma (f^{\w })^{\w }. 
$$ 
Hence, 
$\sigma (f^{\w })^{\w }$ belongs to $\kx _{\alpha }$. 
Therefore, 
we get $\degw \sigma (f^{\w }) =\degw f$. 
\end{proof}

At this point, 
it is useful to remark that 
$\gamma <|\w |_*$ implies $\gamma \leq |\w |-w_{n-1}$ 
for each $\gamma \in \N w_1+\cdots +\N w_{n-1}$, 
and $|\w |_*\leq |\w |+w_1-w_{n-1}$ 
by the definition of $|\w |_*$. 
Under the assumption that $w_1\leq \cdots \leq w_n$, 
we have the following lemma.

\begin{lemma}\label{lem:h_i}
If $\degw \omega <|\w |_*$, 
then $h_1,\ldots ,h_{n-1}$ belong to $k$, 
$h_n$ belongs to $k[\x \sm \{ x_n\} ]$ 
and $\degw h_n\leq w_n-w_{n-1}$. 
\end{lemma}
\begin{proof}
We remark that 
$\degw \eta _i
=|\w |-w_i$ 
for $i=1,\ldots ,n$. 
Hence, 
we have 
$$
\degw h_i+|\w |-w_i=\degw h_i\eta _i
\leq \degw \omega <|\w |_*\leq 
|\w |+w_1-w_{n-1}, 
$$
and so $\degw h_i<w_1-w_{n-1}+w_i$. 
Since $w_1\leq \cdots \leq w_n$ by assumption, 
we know that $\degw h_i<w_1$ 
for $1\leq i\leq n-1$, 
and $\degw h_n<w_n$. 
Thus, 
$h_i$ belongs to $k$ if $i<n$, 
and to $k[\x \sm \{ x_n\} ]$ if $i=n$. 
Since $\degw \eta _n=w_1+\cdots +w_{n-1}$, 
it follows that $\degw h_n\eta _n$ 
belongs to $\N w_1+\cdots +\N w_{n-1}$. 
Since $\degw h_n\eta _n\leq \degw \omega <|\w |_*$, 
this implies that 
$\degw h_n\eta _n\leq |\w |-w_{n-1}$ 
as remarked. 
Therefore, we have 
$\degw h_n\leq w_n-w_{n-1}$. 
\end{proof}

Now, 
we prove Lemma~\ref{lem:ind key}. 
Let $I$ be the set of $1\leq i<n$ 
such that $h_i\neq 0$. 
When $I=\emptyset $, 
the assertion holds for $\phi =\id _{\kx }$, 
$h=h_n$ and $l=n$. 
Hence, 
we may assume that $I\neq \emptyset $. 
We prove the assertion 
by induction on the number of elements of $I$. 
First, 
assume that $I=\{ r\} $ for some $1\leq r<n$. 
Then, 
we have $\omega =h_r\eta _r+h_n\eta _n$ 
with $h_r\neq 0$. 
Since $\degw \omega <|\w |_*$ by assumption, 
we know by Lemma~\ref{lem:h_i} that 
$h_r$ belongs to $k^{\times }$, 
$h_n$ belongs to $k[\x \sm \{ x_n\} ]$ 
and $\degw h_n\leq w_n-w_{n-1}$. 
By integrating $h_r^{-1}h_n$ with respect to $x_r$, 
we obtain $g\in k[\x \sm \{ x_n\} ]$ 
such that $\partial g/\partial x_r=h_r^{-1}h_n$ and 
$$
\degw g=\degw h_nx_r=\degw h_n+w_r
\leq (w_n-w_{n-1})+w_r\leq w_n. 
$$
Define $\phi \in \T_n(k)$ by 
$\phi (x_n)=x_n+(-1)^{n-r}g$ 
and $\phi (x_i)=x_i$ for each $i\neq n$. 
Then, 
we have 
$\degw \phi (f_i)=\degw f_i$ 
for each $i$ by Lemma~\ref{lem:deg pres aut} 
and the preceding inequality. 
Since $h_r$ and $h_n$ are fixed under $\phi $, 
we know by (\ref{eq:omega ^phi note}) 
and (\ref{eq:diff elem}) that 
\begin{align*}
&\omega ^{\phi }
=\phi (h_r)\eta _r^{\phi }+\phi (h_n)\eta _n^{\phi }
=h_r\eta _r^{\phi }+h_n\eta _n 
=h_r\left(\eta _r-\frac{\partial g}{\partial x_r}\eta _n\right)
+h_n\eta _n \\
&\quad =h_r(\eta _r-h_r^{-1}h_n\eta _n)+h_n\eta _n
=h_r\eta _r. 
\end{align*}
Therefore, 
the assertion is true when $\# I=1$.

Next, 
assume that $\#I\geq 2$. 
Let $r,s\in I$ be 
such that $1\leq r<s<n$. 
Since $h_1,\ldots ,h_{n-1}$ belong to $k$ 
by Lemma~\ref{lem:h_i}, 
we can define $\sigma \in \T_n(k)$ by 
$\sigma (x_s)=x_s+(-1)^{s-r}h_r^{-1}h_sx_r$ 
and $\sigma (x_i)=x_i$ for each $i\neq s$. 
Then, 
$\degw \sigma (f)=\degw f$ holds for each $f\in \kx $ 
by Lemma~\ref{lem:deg pres aut}, 
since $\degw h_r^{-1}h_sx_r=w_r\leq w_s$.  
By (\ref{eq:diff elem}), 
we have $\eta _r^{\sigma }=\eta _r-h_r^{-1}h_s\eta _s$ 
and $\eta _i^{\sigma }=\eta _i$ for each $i\neq r$. 
Put $f_i'=\sigma (f_i)$ for $i=1,\ldots ,n-1$. 
Then, 
it follows that 
\begin{align*}
&df_1'\wedge \cdots \wedge df_{n-1}'=\omega ^{\sigma }
=\sigma (h_r)\eta _r^{\sigma }
+\sum _{i\neq r,n}\sigma (h_i)\eta _i^{\sigma }
+\sigma (h_n)\eta _n^{\sigma } \\
&\quad =h_r(\eta _r-h_r^{-1}h_s\eta _s)+\sum _{i\neq r,n}h_i\eta _i
+\sigma (h_n)\eta _n
=\sum _{i\neq s,n}h_i\eta _i
+\sigma (h_n)\eta _n. 
\end{align*}
Since $\degw \sigma (h_n)=\degw h_n$, 
this implies that 
$$
\degw df_1'\wedge \cdots \wedge df_{n-1}'
\leq \degw \omega <|\w |_*. 
$$ 
Thus, 
by induction assumption, 
there exist $\tau \in \T_n(k)$, 
$h\in \kx $ and $1\leq l\leq n$ 
such that 
$$
d\tau (f_1')\wedge 
\cdots \wedge d\tau (f_{n-1}')=h\eta _l
$$
and $\degw \tau (f_i')=\degw f_i'$ 
for $i=1,\ldots ,n-1$. 
Then, 
$\phi :=\tau \circ \sigma $ is a tame automorphism 
such that $\omega ^{\phi }=h\eta _l$. 
Moreover, 
we have $\degw \phi (f_i)=\degw f_i$ 
for each $i$, 
since $\degw \sigma (f_i)=\degw f_i$. 
Therefore, 
the assertion is true when $\# I\geq 2$. 
This completes the proof of Lemma~\ref{lem:ind key}, 
and thereby completing the proof of Theorem~\ref{thm:|w|*}.

\section{Proof of Theorem~\ref{thm:main}}\label{sect:pf main}
\setcounter{equation}{0}

Throughout this section, 
we assume that $n=3$. 
The goal of this section 
is to complete the proof of Theorem~\ref{thm:main}.

\begin{lemma}\label{lem:pf of main}
For any $\w \in (\Gamma _+)^3$ 
and $(d_1,d_2,d_3)\in \mdegw (\Aut _k\kx )$, 
the following assertions hold$:$

\noindent{\rm (i)} 
{\rm (K1)} and {\rm (K2)} imply that 
at least two of 
$d_1$, $d_2$ and $d_3$ 
do not belong to $\{ w_1,w_2,w_3\} $.

\noindent{\rm (ii)} 
{\rm (K1)} and {\rm (K2)} imply that 
$\Delta _{\w }(d_i,d_j)\geq |\w |_*$ for each 
$i,j\in \{ 1,2,3\} $ with $i\neq j$.

\noindent{\rm (iii)} 
If $d_1$ and $d_2$ are linearly dependent over $\Z $, 
then {\rm (K2)} and {\rm (B1)} imply {\rm (B2)}.

\noindent{\rm (iv)} 
If $d_1$ and $d_2$ are linearly dependent over $\Z $, 
then {\rm (K1)}, {\rm (K2)}, {\rm (A)} and {\rm (B2)} 
imply {\rm (K3)}, {\rm (K4)} and {\rm (K5)}.

\noindent{\rm (v)} 
If $d_1$ and $d_2$ are linearly independent over $\Z $, 
then {\rm (K1)}, {\rm (K2)} and {\rm (A)} 
imply {\rm (K3)} and {\rm (K4)}. 
\end{lemma}
\begin{proof}
(i) 
Let $F\in \Aut _k\kx $ be 
such that $\mdegw F=(d_1,d_2,d_3)$. 
By changing the indices of $x_1$, $x_2$ and $x_3$, 
and $w_1$, $w_2$ and $w_3$ if necessary, 
we may assume that $w_1\leq w_2\leq w_3$. 
Then, 
we have $d_i\geq w_i$ for $i=1,2,3$ 
by Lemma~\ref{prop:deg F}, 
since $d_1<d_2<d_3$ by (K1). 
First, 
we show that 
$f_1$ or $f_2$ does not belong to $k[x_1,x_2]$ 
by contradiction. 
Suppose that 
$f_1$ and $f_2$ 
both 
belong to $k[x_1,x_2]$. 
Then, 
$F':=(f_1,f_2)$ belongs to $\Aut _kk[x_1,x_2]$, 
and $f_3=\alpha x_3+p$ 
for some 
$\alpha \in k^{\times }$ and $p\in k[x_1,x_2]$. 
If $(d_1,d_2)\neq (w_1,w_2)$, 
then we have $\degw F'=d_1+d_2>w_1+w_2$, 
since $d_i\geq w_i$ for each $i$. 
Hence, 
$d_i$ belongs to $\N d_j$ for some 
$(i,j)\in \{ (1,2),(2,1)\} $ by Proposition~\ref{prop:vdk}. 
Since $d_1<d_2$, 
it follows that $d_2$ belongs to $\N d_1$, 
contradicting (K2). 
Assume that $(d_1,d_2)=(w_1,w_2)$. 
Then, 
we have $d_3>w_3$, 
since $d_1+d_2+d_3>|\w |$ by (K1). 
Hence, 
we get 
$$
d_3=\degw f_3=\degw {(\alpha x_3+p)}=\degw p. 
$$ 
Since $p$ is an element of $k[x_1,x_2]$, 
it follows that 
$d_3$ belongs to $\Zn w_1+\Zn w_2=\Zn d_1+\Zn d_2$. 
This contradicts (K2). 
Therefore, 
$f_1$ or $f_2$ 
does not belong to $k[x_1,x_2]$. 
Since $w_1\leq w_2\leq w_3$ 
and $d_1<d_2<d_3$, 
this implies that $w_3\leq d_2$. 
Now, 
suppose that (i) is false. 
Then, 
$d_1$ and $d_2$ must be at most $w_3$. 
Hence, 
we have $d_2=w_3$, 
and so $d_1=w_l$ for some $l\in \{ 1,2\} $. 
Thus, 
we may write 
$f_1=\alpha _1x_l+p_1$ 
and $f_2=\alpha _2x_3+p_2$, 
where $\alpha _1,\alpha _2\in k$, 
and $p_1\in k[x_1,\ldots ,x_{l-1}]$ 
and $p_2\in k[x_1,x_2]$ are such that 
$\degw p_1\leq w_l$ and $\degw p_2\leq w_3$. 
We note that $\alpha _2\neq 0$, 
for otherwise $f_1$ and $f_2$ both belong to $k[x_1,x_2]$. 
If $\alpha _1=0$, 
then 
$f_1=p_1$ belongs to $k[x_1,\ldots ,x_{l-1}]$. 
Since $f_1$ is not a constant, 
it follows that $l=2$, 
and $f_1=\alpha x_1+\beta $ 
for some $\alpha \in k^{\times }$ 
and $\beta \in k$. 
Hence, 
we may assume that $\alpha _1\neq 0$ 
by changing $l$ if necessary. 
Since $k[x_l-p_1,x_m]=k[x_1,x_2]$ 
holds for $m\in \{ 1,2\} $ with $m\neq l$, 
we can define $\sigma \in \Aut _kk[x_1,x_2]$ by 
$\sigma (x_l)=\alpha _1^{-1}(x_l-p_1)$ 
and $\sigma (x_m)=x_m$. 
Then, 
$\sigma $ preserves the $\w $-degree 
of each element of $k[x_1,x_2]$ 
by Lemma~\ref{lem:deg pres aut}. 
Since $p_2$ is an element of $k[x_1,x_2]$, 
we can extend $\sigma $ to an element of $\Aut _k\kx $ 
by setting $\sigma (x_3)=\alpha _2^{-1}(x_3-\sigma (p_2))$. 
Then, 
$\sigma $ also preserves 
the $\w $-degree of each element of $\kx $, 
since $\degw \sigma (p_2)=p_2\leq w_3$. 
By definition, 
we have $\sigma (f_1)=x_l$ and $\sigma (f_2)=x_3$, 
and so $\sigma (f_3)=\alpha x_m+p$ 
for some $\alpha \in k^{\times }$ and $p\in k[x_l,x_3]$. 
Since $\degw x_m=w_m\leq w_3=d_2<d_3$, 
we get 
$$
d_3=\degw f_3=\degw \sigma (f_3)
=\degw {(\alpha x_m+p)}=\degw p. 
$$
Therefore, 
$d_3$ belongs to 
$\Zn w_l+\Zn w_3=\Zn d_1+\Zn d_2$, 
contradicting (K2).

(ii) Take any $i,j\in \{ 1,2,3\} $ with $i\neq j$. 
Then, 
$d_i$ or $d_j$ 
does not belong to $\{ w_1,w_2,w_3\} $ by (i). 
Moreover, 
since $d_1<d_2<d_3$ by (K1), 
we see from (K2) that 
$d_i$ and $d_j$ do not belong to $\N d_j$ and $\N d_i$, 
respectively. 
Therefore, 
we have $\Delta _{\w }(d_i,d_j)\geq |\w |_*$ 
by the remark after Theorem~\ref{thm:|w|*}.

(iii) Since $d_1$ and $d_2$ 
are linearly dependent over $\Z $ by assumption, 
we may write $d_i=u_id$ for $i=1,2$, 
where $u_1,u_2\in \N $ 
are such that $\gcd (u_1,u_2)=1$, 
and $d=\gcd (d_1,d_2)$. 
Now, 
suppose that (B2) is not satisfied. 
Then, 
we have 
$$
d_3\geq 
\lcm (d_1,d_2)-d_1-d_2+|\w |_*
=(u_1u_2-u_1-u_2)d+|\w |_*. 
$$
Since $|\w |_*\geq d$ and $d_3=u_3d$ 
for some $u_3\in \N $ by (B1), 
it follows that 
$$
u_3\geq (u_1u_2-u_1-u_2)+1=(u_1-1)(u_2-1). 
$$
Hence, 
$u_3$ belongs to $\Zn u_1+\Zn u_2$ 
by Sylvester's formula for the Frobenius number. 
Therefore, 
$d_3$ belongs to $\Zn d_1+\Zn d_2$, 
contradicting (K2).

(iv) 
First, we check (K3) and (K4). 
By (ii), 
we have 
\begin{equation}\label{eq:delta bound}
\max \{ \Delta _{\w }(d_i,d_j),|\w |_*\} 
=\Delta _{\w }(d_i,d_j)
\end{equation}
for each $i\neq j$. 
Hence, 
(A) implies (K3) and (K4) 
if $3d_2\neq 2d_3$ or $sd_1\neq 2d_3$ 
for any odd number $s\geq 3$. 
Assume that $3d_2=2d_3$ and $sd_1=2d_3$ 
for some odd number $s\geq 3$. 
Then, 
there exists $d\in \Gamma _+$ 
such that $d_1=6d$, $d_2=2sd$ and $d_3=3sd$. 
Since $d_2$ does not belong to $\N d_1$ by (K2), 
we have $3\nmid s$, 
and hence $\lcm (d_1,d_2)=6sd=2d_3$. 
Thus, 
(B2) implies that 
$$
d_1+d_2<\lcm (d_1,d_2)-d_3+|\w |_*
=d_3+|\w |_*
\leq d_3+\Delta _{\w }(d_i,d_3) 
$$ 
for $i=1,2$ in view of (ii). 
Therefore, 
we get (K3) and (K4).

To show (K5), 
we may assume that $4d_1=3d_2$. 
Write $d_1=3d$ and $d_2=4d$, 
where $d=\gcd (d_1,d_2)$. 
Then, 
we have 
\begin{equation}\label{eq:lemma:5d}
d_3<\lcm (d_1,d_2)-d_1-d_2+|\w |_*
=(3\cdot 4-3-4)d+|\w |_*
=5d+|\w |_* 
\end{equation}
by (B2). 
Hence, 
it suffices to show that 
$|\w |_*\leq \Delta _{\w }(2d_1,d_2)$. 
By the remark after Theorem~\ref{thm:|w|*}, 
we have only to check that 
$2d_1$ and $d_2$ do not belong to 
$\N d_2$ and $\N 2d_1$, 
respectively, 
and that $2d_1$ or $d_2$ 
does not belong to $\{ w_1,w_2,w_3\} $. 
Since $2d_1=6d$ and $d_2=4d$, 
the first part is clear. 
We prove the last part by contradiction. 
Suppose that 
$2d_1$ and $d_2$ 
both belong to $\{ w_1,w_2,w_3\} $. 
We may assume that $w_1\leq w_2\leq w_3$ 
as in the proof of (i). 
Then, 
$(2d_1,d_2)$ is equal to $(w_2,w_1)$ or $(w_3,w_2)$, 
since $2d_1=6d>4d=d_2$. 
We claim that the former case does not occur, 
for otherwise $w_i\geq w_1=d_2>d_1$ for any $i$, 
contradicting Lemma~\ref{prop:deg F}. 
Hence, 
we have $d_2=w_2$. 
Since $d_1<d_2$, 
this implies that $f_1$ belongs to $k[x_1]$, 
and hence is a linear polynomial 
in $x_1$ over $k$. 
Thus, 
we get $d_1=w_1$. 
Therefore, 
$d_1$ and $d_2$ belong to 
$\{ w_1,w_2,w_3\} $, 
contradicting (i).

(v) 
By (K1) and (K2), 
we have (\ref{eq:delta bound}). 
Since $d_1$ and $d_2$ are linearly independent over $\Z $, 
we have $3d_2\neq 2d_3$ or $sd_1\neq 2d_3$ 
for any odd number $s\geq 3$. 
Hence, 
(A) implies (K3) and (K4). 
\end{proof}

Now, 
let us prove Theorem~\ref{thm:main}. 
Suppose to the contrary that 
there exists $F\in \T_3(k)$ such that 
$\mdegw F=(d_1,d_2,d_3)$. 
First, 
assume that $d_1$ and $d_2$ 
are linearly independent over $\Z $. 
Then, 
we have (K3) and (K4) 
by Lemma~\ref{lem:pf of main} (v). 
Hence, 
$(d_1,d_2,d_3)$ does not belong to 
$\mdegw \T_3(k)$ by Theorem~\ref{prop:main} (ii), 
a contradiction. 
Next, 
assume that $d_1$ and $d_2$ 
are linearly dependent over $\Z $. 
Then, 
we have (B1) or (B2). 
On account of Lemma~\ref{lem:pf of main} (iii), 
we may assume that (B2) is satisfied. 
Then, 
we get (K3), (K4) and (K5) 
by Lemma~\ref{lem:pf of main} (iv). 
Since $F$ is tame by supposition, 
we have 
$$
\degw df_1\wedge df_2
\geq \Delta _{\w }(d_1,d_2)\geq |\w |_* 
$$
by Lemma~\ref{lem:pf of main} (ii). 
Hence, 
it follows from (B2) that 
$$
d_1+d_2+d_3
<\lcm (d_1,d_2)+|\w |_*
\leq \lcm (d_1,d_2)+\degw df_1\wedge df_2. 
$$
Thus, 
$F$ does not belong to $\T_3(k)$ 
by Theorem~\ref{prop:main} (i), 
a contradiction. 
Therefore, 
$(d_1,d_2,d_3)$ does not belong to $\mdegw \T_3(k)$. 
This completes the proof of Theorem~\ref{thm:main}. 

\section{Applications}
\label{sect:appl}
\setcounter{equation}{0}

In this section, 
we give applications of our theorems. 
Throughout, 
$d_1$, $d_2$ and $d_3$ 
denote positive integers. 
First, 
we briefly discuss when the assumptions of 
Theorem~\ref{thm:total degree} are fulfilled. 
Set 
$$
d_i'=\frac{d_i}{\gcd (d_1,d_2,d_3)}
$$
for $i=1,2,3$. 
Clearly, 
$d_i'$ is an odd number 
if so is $d_i$ for each $i$.

In this notation, 
we have the following lemma.

\begin{lemma}\label{lem:a total}
Assume that $d_1,d_2,d_3\in \N $ 
satisfy $d_1\leq d_2\leq d_3$ and {\rm (c)}. 
Then, 
we have {\rm (a)} 
if one of the following conditions holds$:$

\nd {\rm (1)} 
$d_1'$ is an odd number, 
and $d_2'$ is an odd number 
or $d_3'\not\equiv 0\pmod{3}$.

\nd {\rm (2)} 
$d_1\neq 2\gcd (d_1,d_3)$ and 
$d_2'$ is an odd number.

\nd {\rm (3)} 
$d_1'\equiv 0\pmod{2^l}$ for some $l\geq 2$, 
and $d_2'$ and $d_3'$ are odd numbers.

\nd {\rm (4)} 
$d_3$ is a prime number.

\nd {\rm (5)} 
$d_3-d_2\geq d_1-2$.

\nd {\rm (6)} 
$d_1'$ is an odd number, 
and $3d_2\neq 2d_3$ or $2d_1\leq d_2+5$. 
\end{lemma}
\begin{proof}
It is easy to see that both (1) and (2) imply (a1). 
In the case of (3), 
we have $3d_2\neq 2d_3$ 
since $d_2'$ is an odd number. 
Since 
$d_1'\equiv 0\pmod{2^l}$ for some $l\geq 2$, 
and $d_3'$ is an odd number, 
$sd_1\neq 2d_3$ holds for any odd number $s$. 
Thus, 
we get (a1). 
We prove that (4) implies (a1) by contradiction. 
If $3d_2=2d_3$, then we have $d_2=2$. 
Since $d_1\leq d_2\leq d_3$, 
this implies that $d_1=1$ or $d_1=d_2$, 
contradicting (c). 
If $sd_1=2d_3$ for some odd number $s\geq 3$, 
then we have $d_1=2$. 
Hence, 
$d_2$ and $d_3$ must be odd numbers in view of (c). 
Since $d_2\leq d_3$, 
it follows that 
$d_3=2l+d_2
=ld_1+d_2$ for some $l\in \Zn $, 
contradicting (c). 
Therefore, 
(4) implies (a1). 
Clearly, 
(5) is equivalent to (a2). 
Finally, 
we consider the case (6). 
Since $d_1'$ is an odd number, 
we have (a1) if $3d_2\neq 2d_3$. 
So assume that $3d_2=2d_3$. 
Then, 
$d_2$ is an even number. 
Since $2d_1\leq d_2+5$ by assumption, 
it follows that $2d_1\leq d_2+4$. 
Thus, 
by noting $3d_2=2d_3$, 
we get 
$d_1\leq d_2/2+2=d_3-d_2+2$. 
Therefore, 
(a2) is satisfied. 
\end{proof}

By means of Lemma~\ref{lem:a total}, 
we get the following corollary to 
Theorem~\ref{thm:total degree}. 
Here, we note that 
the assumption (\ref{eq:b}) is fulfilled 
whenever $\gcd (d_1,d_2)=1$.

\begin{corollary}\label{cor:karas general}
Let $d_1,d_2,d_3\in \N $ be such that $d_1\leq d_2\leq d_3$, 
and 
\begin{equation}\label{eq:b}
\gcd (d_1,d_2,d_3)=\gcd (d_1,d_2)\leq 3
\quad \text{or}\quad 
d_1+d_2+d_3\leq \lcm (d_1,d_2)+2. 
\end{equation}
Assume that one of {\rm (1)} through {\rm (6)} 
in Lemma~{\rm \ref{lem:a total}} is satisfied. 
Then, 
$(d_1,d_2,d_3)$ belongs to $\mdeg \T_3(k)$ 
if and only if 
$d_1\mid d_2$ or $d_3$ belongs to $\Zn d_1+\Zn d_2$. 
\end{corollary}

The ``only if" part of Corollary~\ref{cor:karas general} 
readily follows 
from Theorem~\ref{thm:total degree}, 
while the ``if" part is due to 
Kara\'s~\cite[Proposition~2.2]{345} 
(cf.~Proposition~\ref{prop:Karas345}).

In the following, 
we prove a variety of Kara\'s type theorems 
using our main results. 
The ``if" parts of Corollaries~\ref{cor:KZ} 
through \ref{cor:tousa2} below 
follow from the proposition of Kara\'s as above. 
So we only check 
(the contrapositions of) 
the ``only if" parts. 
Corollaries~\ref{cor:KZ} through \ref{cor:LiDu1} 
are special cases of Corollary~\ref{cor:karas general}. 
We prove Corollaries~\ref{cor:tousa}, \ref{cor:tousa2} and \ref{cor:jc} 
using Theorem~\ref{thm:total degree}, 
and Corollary~\ref{cor:Kanehira} 
using Theorem~\ref{thm:main}.

The following result is given by 
Kara\'s-Zygad\l o~\cite[Theorem 2.1]{KZ}. 
This is a special case of 
(1) of Corollary~\ref{cor:karas general}.

\begin{corollary}[Kara\'s-Zygad\l o]\label{cor:KZ}
Let $d_3\geq d_2\geq d_1\geq 3$ be integers 
such that $d_1$ and $d_2$ 
are odd numbers with $\gcd (d_1,d_2)=1$. 
Then, 
$(d_1,d_2,d_3)$ belongs to $\mdeg \T_3(k)$ 
if and only if $d_3$ belongs to $\Zn d_1+\Zn d_2$. 
\end{corollary}

The following result is given by 
Kara\'s~\cite[Theorem 1.1]{3dd}.

\begin{corollary}[Kara\'s]\label{cor:Karas3dd}
Let $d_3\geq d_2\geq 3$ be integers. Then, 
$(3,d_2,d_3)$ belongs to $\mdeg \T_3(k)$ 
if and only if 
$3\mid d_2$ or $d_3$ belongs to $\Zn 3+\Zn d_2$. 
\end{corollary}

This result of Kara\'s was generalized by 
Sun-Chen~\cite[Theorem 3.3]{SunChen} 
as follows. 
Here, 
we mention that 
Kara\'s~\cite[Theorem 7.1]{Karas thesis} 
independently obtained a similar result 
in which (iii) is replaced by 
the stronger assumption that $d_2\geq 2d_1-3$.

\begin{corollary}[Sun-Chen]
\label{cor:SunChen}
Let $d_3\geq d_2\geq d_1\geq 3$ be integers 
with $d_1$ a prime number. 
Assume that 
one of the following conditions holds$:$ 

{\rm (i)} $d_2/\gcd (d_2,d_3)\neq 2;$ \quad 
{\rm (ii)} $d_3/\gcd (d_2,d_3)\neq 3;$ \quad 
{\rm (iii)} $d_2\geq 2d_1-5$. \\
Then, 
$(d_1,d_2,d_3)$ belongs to $\mdeg \T_3(k)$ 
if and only if 
$d_1\mid d_2$ or $d_3$ belongs to $\Zn d_1+\Zn d_2$. 
\end{corollary}

The ``only if" part of Corollary~\ref{cor:SunChen} 
follows from (6) of Corollary~\ref{cor:karas general}. 
In fact, 
(i) or (ii) holds if and only if $3d_2\neq 2d_3$, 
and $d_1$ is an odd prime number by assumption. 
Moreover, 
we have $\gcd (d_1,d_2)=1$ 
if $d_1\nmid d_2$.

The following result is given by 
Kara\'s~\cite[Theorems 6.2 and 6.10]{Karas thesis}.

\begin{corollary}[Kara\'s]
\label{cor:Karas4dd}
{\rm (i)} 
Let $d_2$ and $d_3$ be odd numbers with 
$d_3\geq d_2\geq 5$. 
Then, 
$(4,d_2,d_3)$ belongs to $\mdeg \T_3(k)$ 
if and only if 
$d_3$ belongs to $\Zn 4+\Zn d_2$.

\nd {\rm (ii)} 
Let $d_2$ and $d_3$ 
be an odd number and an even number, 
respectively, 
such that 
$d_3\geq d_2\geq 5$ 
and $d_3-d_2\neq 1$. 
Then, 
$(4,d_2,d_3)$ belongs to $\mdeg \T_3(k)$ 
if and only if 
$d_3$ belongs to $\Zn 4+\Zn d_2$. 
\end{corollary}

In both (i) and (ii) of 
Corollary~\ref{cor:Karas4dd}, 
we have $\gcd (4,d_2)=1$ 
since $d_2$ is an odd number. 
Hence, 
(i) is a special case of (3) of 
Corollary~\ref{cor:karas general}. 
The ``only if" part of (ii) follows from 
(5) of Corollary~\ref{cor:karas general}. 
In fact, 
if $d_3$ does not belong to $\Zn 4+\Zn d_2$, 
then 
we have $d_3-d_2\geq 2=4-2$, 
since $d_3\geq d_2$ and $d_3-d_2\neq 1$ 
by assumption.

The following result is given by 
Li-Du~\cite[Theorems 3.3 and 4.2]{LiDu1}.

\begin{corollary}[Li-Du]
\label{cor:LiDu1}
Let $d_3\geq d_2\geq d_1\geq 3$ be integers. 

\noindent
{\rm (i)} 
When $d_1/\gcd (d_1,d_3)\neq 2$ 
and $d_2$ is a prime number, 
$(d_1,d_2,d_3)$ belongs to $\mdeg \T_3(k)$ 
if and only if 
$d_1=d_2$ or $d_3$ belongs to $\Zn d_1+\Zn d_2$.

\noindent{\rm (ii)} 
When $\gcd (d_1,d_2)=1$ 
and $d_3$ is a prime number, 
$(d_1,d_2,d_3)$ belongs to $\mdeg \T_3(k)$ 
if and only if 
$d_3$ belongs to $\Zn d_1+\Zn d_2$. 
\end{corollary}

The ``only if" part of 
Corollary~\ref{cor:LiDu1} (i) 
follows from (2) of Corollary~\ref{cor:karas general}, 
since $d_2$ is an odd prime number, 
and so $\gcd (d_1,d_2)=1$ if $d_1\neq d_2$. 
(ii) is a special case of 
(4) of Corollary~\ref{cor:karas general}.

The following result is given by 
Li-Du~\cite[Theorem 3.3]{LiDu2}.

\begin{corollary}[Li-Du]\label{cor:tousa}
Let $a,d\in \N $ be such that $a\geq 3$. 
Assume that $4d\neq ta$ for any odd number $t\geq 1$. 
Then, 
$(a,a+d,a+2d)$ belongs to $\mdeg \T_3(k)$ 
if and only if $a\mid 2d$. 
\end{corollary}

We can slightly extend this result as follows.

\begin{corollary}\label{cor:tousa2}
Let $a,d\in \N $ be such that $a\geq 3$. 
Assume that $a=4l$ and $d=tl$ 
for some $l\in \N $ and an odd number $t\geq 1$ 
with $(t-4)l+2\geq 0$. 
Then, 
$(a,a+d,a+2d)$ belongs to $\mdeg \T_3(k)$ 
if and only if $a\mid 2d$. 
\end{corollary}

We prove the ``only if" parts of 
Corollaries~\ref{cor:tousa} and \ref{cor:tousa2} 
together by means of 
Theorem~\ref{thm:total degree}. 
Set $(d_1,d_2,d_3):=(a,a+d,a+2d)$. 
Assume that $a\nmid 2d$. 
Then, 
we have $a\nmid d$, 
and so $d_1\nmid d_2$. 
We show that 
$d_3$ does not belong to $\Zn d_1+\Zn d_2$ 
by contradiction. 
Suppose that $a+2d=b_1a+b_2(a+d)$ 
for some $b_1,b_2\in \Zn $. 
Then, 
we have $b_2\leq 1$, 
since $a$ and $d$ are positive. 
If $b_2=1$, 
then it follows that $d=b_1a$, 
and hence $a\mid d$, 
a contradiction. 
If $b_2=0$, 
then we have $2d=(b_1-1)a$, 
and hence $a\mid 2d$, 
a contradiction. 
Thus, 
$d_3$ does not belong to $\Zn d_1+\Zn d_2$. 
Therefore, 
(c) is satisfied. 
Next, 
we show (b2). 
Since $a\nmid d$, 
we have $a/\gcd (a,d)\geq 2$. 
We claim that $a/\gcd (a,d)\geq 3$. 
In fact, 
if $a=2\gcd (a,d)$, 
then $a$ divides $2d$, 
a contradiction. 
Thus, 
we get 
$$
\lcm (d_1,d_2)
=\frac{d_1d_2}{\gcd (d_1,d_2)}
=\frac{a(a+d)}{\gcd (a,d)}
\geq 3(a+d)=d_1+d_2+d_3. 
$$
This implies (b2). 
Finally, 
we show (a). 
In the case of 
Corollary~\ref{cor:tousa}, 
$sd_1=sa$ is not equal to $2d_3=2(a+2d)$ 
for any odd number $s\geq 3$, 
for otherwise $4d=(s-2)a$. 
Since $a\neq d$, 
we get $3d_2=3(a+d)\neq 2(a+2d)=2d_3$. 
Thus, 
(a1) is satisfied. 
In the case of Corollary~\ref{cor:tousa2}, 
we have 
$$
d_1+d_2=4l+(4l+tl)
\leq 8l+tl+\Bigl((t-4)l+2\Bigr)
=4l+2tl+2=d_3+2. 
$$
Hence, 
(a1) is satisfied. 
Therefore, 
$(d_1,d_2,d_3)$ does not belong to $\mdeg \T_3(k)$ 
by Theorem~\ref{thm:total degree}. 
This proves the ``only if" parts of 
Corollaries~\ref{cor:tousa} and \ref{cor:tousa2}.

The following corollary is a generalization of 
Kara\'s~\cite[Theorem 7.9]{Karas thesis} and 
Sun-Chen~\cite[Remark 3.9]{SunChen}. 
We mention that 
Kara\'s originally proved the case 
where $d$ is a prime number with $5\leq d\leq 31$, 
while 
Sun-Chen deduced from Corollary~\ref{cor:SunChen} 
that the assertion holds whenever 
$d$ is a prime number with $d\geq 5$ 
(Sun-Chen also included the case where $d=3$, 
but this is an obvious mistake in view 
of Proposition~\ref{prop:Karas345}).

\begin{corollary}\label{cor:jc}
If $d\geq 5$ and $d\neq 6,8$ for $d\in \N $, 
then $(d,2(d-2),3(d-2))$ does not belong to $\mdeg \T_3(k)$. 
\end{corollary}
\begin{proof}
Set $(d_1,d_2,d_3):=(d,2(d-2),3(d-2))$. 
We prove the assertion by means of 
Theorem~\ref{thm:total degree}. 
Since 
$$
d_1+d_2=3d-4=3(d-2)+2=d_3+2, 
$$
we have (a2). 
We check (c). 
Since $d\geq 5$, 
we see that $d_2=2(d-2)$ 
is not divisible by $d_1=d$. 
We show that 
$d_3$ does not belong to $\Zn d_1+\Zn d_2$ 
by contradiction. 
Suppose that 
$3(d-2)=a_1d+a_2\cdot 2(d-2)$ for some $a_1,a_2\in \Zn $. 
Then, 
we have $a_2\leq 1$, 
since $d\geq 5$. 
If $a_2=1$, 
then it follows that $(1-a_1)d=2$. 
This contradicts that 
$a_1$ and $d$ are integers with $d\geq 5$. 
If $a_2=0$, 
then we have $(3-a_1)d=6$. 
Since $d\neq 6$ and $d\geq 5$, 
we get a contradiction similarly. 
Thus, 
$d_3$ does not belong to $\Zn d_1+\Zn d_2$, 
proving (c). 
Finally, 
we show (b2). 
We claim that $d\geq 3\gcd (d,4)$. 
This is clear if $d\geq 12$. 
Since $d\geq 5$ and $d\neq 6,8$, 
the rest of the cases are checked easily. 
Hence, 
we have 
$$
\lcm (d_1,d_2)
=\frac{d_1d_2}{\gcd (d_1,d_2)}
=\frac{d\cdot 2(d-2)}{\gcd (d,4)}
\geq 3\cdot 2(d-2)=d_1+d_2+d_3-2, 
$$
proving (b2). 
Therefore, 
$(d_1,d_2,d_3)$ does not belong to $\mdeg \T_3(k)$ 
thanks to Theorem~\ref{thm:total degree}. 
\end{proof}

The following result of 
Kanehira~\cite[Theorem 28]{Kanehira} 
(cf.~\cite[Theorem 3.1]{Sep}) 
is derived from Theorem~\ref{thm:main}.

\begin{corollary}[Kanehira]
\label{cor:Kanehira}
Let $d_3\geq d_2>d_1\geq 3$ be integers 
such that $d_1$ and $d_2$ 
are odd numbers with $\gcd (d_1,d_2)=1$. 
If there exist $\w \in \N ^3$ 
and $F\in \T_3(k)$ such that $\mdegw F=(d_1,d_2,d_3)$
and $\degw F>|\w |$, 
then $d_3$ belongs to $\Zn d_1 +\Zn d_2 $.
\end{corollary}
\begin{proof}
By assumption, 
$(d_1,d_2,d_3)$ is an element of $\mdegw \T _3(k)$ 
such that $d_1<d_2\leq d_3$ and 
$d_1+d_2+d_3>|\w |$. 
Suppose to the contrary that 
$d_3$ does not belong to $\Zn d_1 +\Zn d_2 $. 
Then, 
we have $d_2\neq d_3$, 
and so $d_2<d_3$. 
Thus, 
(K1) is satisfied. 
Since $d_1\geq 3$ and $\gcd (d_1,d_2)=1$, 
we see that 
$d_2$ does not belong to $\N d_1$. 
Hence, 
we get (K2) by supposition. 
Since $d_1$ and $d_2$ are odd numbers, 
(A1) is clear. 
Since $\gcd (d_1,d_2)=1$, 
we see that (B1) is satisfied. 
Thus, 
$(d_1,d_2,d_3)$ does not belong to $\mdegw \T_3(k)$ 
by Theorem~\ref{thm:main}. 
This is a contradiction. 
Therefore, 
$d_3$ belongs to $\Zn d_1 +\Zn d_2 $. 
\end{proof}

As noted after Theorem~\ref{thm:|w|*}, 
$\Delta _{\w }(4,6)\geq 4$ holds when 
$\Gamma =\Z $ and $\w =(1,1,1)$ 
by Kara\'s~\cite[Theorem 14]{456}. 
Modulo this result, 
Theorem~\ref{thm:main} 
implies Kara\'s~\cite[Theorem 1]{456} 
asserting that 
$(4,5,6)$ does not belong to $\mdeg \T_3(k)$. 
In fact, 
$(d_1,d_2,d_3)=(4,5,6)$ satisfies (K1) and (K2) 
for $\Gamma =\Z $ and $\w =(1,1,1)$. 
Since $3d_1=12=2d_3$, 
and $d_1+d_2=9$ is less than $d_3+\Delta _{\w }(4,6)=10$, 
we see that (A3) holds for $s=3$. 
Since $d_1+d_2+d_3=15$ is less than 
$\lcm (d_1,d_2)=20$, 
we get (B2).

\noindent
Department of Mathematics and Information Sciences\\ 
Tokyo Metropolitan University \\
1-1  Minami-Osawa, Hachioji \\
Tokyo 192-0397, Japan\\
kuroda@tmu.ac.jp

\end{document}